\documentclass[11pt, a4paper]{article}
\usepackage {fullpage}						
\usepackage {ucs}
\usepackage [utf8x]{inputenc}
\usepackage {graphicx}
\usepackage [bookmarks=false, pdfstartview={FitH}, colorlinks=true, linkcolor=blue, citecolor=blue, urlcolor=blue]{hyperref} 
\usepackage {amsfonts,amsmath,amsthm,amssymb}
\usepackage {latexsym}
\usepackage {graphics, graphicx}
\usepackage {fancyhdr}
\usepackage[table]{xcolor}
\usepackage {epstopdf} 						
\usepackage {multicol}						
\usepackage {multirow}  					
\usepackage {soul} 							
\usepackage {authblk} 						
\usepackage {datetime} 						
\usepackage {subfigure}						
\usepackage {parskip}						
\usepackage {setspace}						
\usepackage[font=small,textfont=it,width=0.97\textwidth]{caption}		
\usepackage{mathtools}						

\newtheorem{theorem}{Theorem}[section]

\theoremstyle{remark}
	\newtheorem{remark}{Remark}[section]

\renewcommand{\arraystretch}{1.28}
\allowdisplaybreaks[1]

\numberwithin{equation}{section}
\numberwithin{figure}{section}
\numberwithin{table}{section}

\newcommand{\eps}{\varepsilon}

\newcommand{\Vector}[1]{\begin{pmatrix} #1 \end{pmatrix}}
\newcommand{\Matrix}[1]{\begin{bmatrix} #1 \end{bmatrix}}

\newcommand{\minmod}{\operatorname{minmod}}
\renewcommand{\mod}{\operatorname{mod}}
\newcommand{\divergence}{\operatorname{div}}
\newcommand{\w}{\mathbf w}
\newcommand{\x}{\mathbf x}
\newcommand{\e}{\mathbf e}
\renewcommand{\a}{\mathbf a}
\newcommand{\bnu}{ \boldsymbol{\nu} }
\def\softd{{\leavevmode\setbox1=\hbox{d}%
          \hbox to 1.05\wd1{d\kern-0.4ex{\char039}\hss}}}

\newcommand {\startenv} {\vskip 0.05em
\begin{tabular}{||l}\parbox[t]{0.95\linewidth}}
\newcommand {\stopenv} {\end{tabular}\vskip 0.05em}

\title{Numerical study of cancer cell invasion dynamics using adaptive mesh refinement: the urokinase model}
	\author{
	Niklas Kolbe
		\thanks{Institute of Mathematics, University of Mainz, Germany\hfill {\tt kolbe@uni-mainz.de}}~{},
	Jana Kat'uchov\'a
		\thanks{Faculty of Medicine, University of Ko\v{s}ice, Slovakia \hfill {\tt katuchova@gmail.com}}~{},
	Nikolaos Sfakianakis
		\thanks{Institute of Mathematics, University of Mainz, Germany \hfill {\tt sfakiana@uni-mainz.de}}~{},
	Nadja Hellmann
		\thanks{Institute of Molecular Biophysics, University of Mainz, Germany \hfill {\tt nhellman@uni-mainz.de}}~{},
	and \\
	M\'aria Luk\'{a}\v{c}ov\'{a}-Medvi\softd ov\'{a}
		\thanks{Institute of Mathematics, University of Mainz, Germany\hfill {\tt lukacova@uni-mainz.de}}
	}

\date{\today} 

\begin{document}
\maketitle
\begin{abstract}
In the present work we investigate the chemotactically and proteolytically driven tissue invasion by cancer cells. The model employed is a system of advection-reaction-diffusion equations that features the role of the serine protease urokinase-type plasminogen activator. The analytical and numerical study of this system constitutes a challenge due to the merging, emerging, and travelling concentrations that the solutions exhibit.

Classical numerical methods applied to this system necessitate very fine discretization grids to resolve these dynamics in an accurate way. To reduce the computational cost without sacrificing the accuracy of the solution, we apply adaptive mesh refinement techniques, in particular h-refinement. Extended numerical experiments exhibit that this approach provides with a higher order, stable, and robust numerical method for this system. We elaborate on several mesh refinement criteria and compare the results with the ones in the literature.

We prove, for a simpler version of this model, $L^\infty$ bounds for the solutions, we study the stability of its conditional steady states, and conclude that it can serve as a test case for further development of mesh refinement techniques for cancer invasion simulations.
\end{abstract}



\section{Introduction}\label{intro}
Cancer is one of the most frequent causes of death worldwide. As reported in \cite{katuchova3} by 2020  about 70\% of all cancer-related death will occur in developing countries, where the survival rate is only about 20\%-30\% due to a late diagnosis. The development of cancer involves different  sub-processes like growth, vascularization, tissue invasion, and metastasis, \cite{jana}.

Since the 1950s, the mathematical description of complex biological system has gained increasing interest and has become a very active research field, e.g. \cite{Nordling.1953, Armitage.1954, Fisher.1958}. The growth and invasion of cancer was subject of many theoretical studies, concentrating on different aspects and employing different mathematical approaches; see for example cancer growth \cite{Alt.1985,Preziosi.2003,Deutsch.2013}; cancer cells invasion of the Extracellular Matrix (ECM) \cite{Perumpanani.1996, Anderson.2000, Turner.2002, Frieboes.2006, Gerisch.2008, Poplawski, Szymanska.2009, Painter.2010}; cancer stem cells modelling \cite{Gupta.2011, Czochra.2012}, to name but a few.

This is the first in a series of papers where we address a mathematical approach to describe the invasion of cancer cells into the extracellular matrix and which involve the degeneration of the adjacent tissue by the cancer cells and their migration into nearby areas. In these works we model, analyse, and numerically resolve different biological theories that address the pathway of chemical interactions taking place during the invasion of the cancer cells. 

In the current paper, in particular, our primer objective is to lay the foundation and to propose/present the numerical treatment that we use in our studies. We focus on a model introduced by Chaplain and Lolas in \cite{lolas2003phd,chaplain2005mathematical} and further analyzed in \cite{andasari2011mathematical,gerisch2006robust,Gerisch.2008, Szymanska.2009, kolbe2013master}. It describes the dynamics of cancer invasion using a deterministic model of macroscopic quantities. The invasion process is dominated by directed cell movement due to the gradient of extracellular chemicals (chemotaxis) and due to the gradient in the structure of the extracellular matrix (haptotaxis) that is mathematically described according to Patlak and Keller-Segel \cite{patlak1953random, keller1971model}. The model includes further the interactions of the cancer cells with different proteins and diffusion of the cells and the proteins.

Due to the model dynamics, suitable, accurate numerical methods of high order are needed for the simulations. To this end we employ, at first, a finite-volume method introduced in \cite{kurganovnumerical}, augmented with several time integration methods. We notice in one-dimensional experiments that a large number of grid cells is needed to properly resolve the dynamics, albeit the methods employed are stable and second order accurate. The corresponding two-dimensional experiments reveal similar dynamics. Using though a grid as fine as in the one-dimensional case, renders the computation prohibitively expensive.

Such difficulties are not new in the mathematical literature; there are several examples where the convergence of a numerical method depends heavily on the size of the discretization cells, and at the same time, the use of uniformly fine discretization grids is not satisfactory due to the increased computational cost. In such cases, mesh refinement techniques are often seen as an alternative numerical treatment. With such methods, one alters the local density of the discretization grid either by refining the mesh locally. It has been seen, time and again, that such methods can improve the quality of the numerical solutions, and at the same time reduce the computational cost, see e.g. \cite{Sfakianakis.2010, kroner2000posteriori, Sfakianakis.2013a, ohlberger19999adaptive, puppo2012adaptive, Sfakianakis.2013b}. Therefore, we investigated the properties of the cancer-invasion-model when employing these mesh refinement techniques, in particular h-refinement in the form of cell bisection based on properly chosen estimator functions.

We study analytically a reduced chemotaxis-haptotaxis model with logistic growth, and compare it with the original system in terms of qualitative behaviour of their respective solutions. Moreover, we investigate the conditional stability of the steady states for particular parameter ranges and justify the similarities in the transient behaviour of merging, emerging and travelling concentrations that both systems exhibit numerically. Furthermore we prove $L^\infty$ bounds on the solutions of both systems; allowing hence to use the smaller model and the corresponding parameter set, as a test case for the development of the mesh refinement techniques for the cancer invasion models.

The present paper is organized as follows: In Section \ref{bio} we describe in more details some important physiological aspects of tumor biology. Based on the biological understanding we explain in Section \ref{model} the corresponding mathematical model describing tumor proliferation and invasion. In Sections \ref{methods} and \ref{experiments} we describe the numerical methods, and mesh refinement techniques, which we have used for the approximations and discuss the results of numerical experiments.

\section{Tumor Biology}\label{bio}

Tumorgenesis is a multistep process, in which normal cells progressively convert to cancer cells. This process is associated with various changes in cell physiology common to most of the cancers. These are, in particular, self-sufficiency in growth signals, insensitivity to inhibitory growth signals, evasion of programmed cell death, limitless replicative potential, sustained angiogenesis, tissue invasion and metastasis, and immunoediting \cite{katuchova1, katuchova2, katuchova3}.

Tumor development is directly and indirectly influenced by paracrine as well as autocrine signals. Such factors include angiogenesis factors, growth factors, chemokines (signaling molecules originally characterized by their ability to induce chemotaxis), cytokines, hormones, enzymes, cytolytic factors, and so forth, which may promote or reduce tumor growth \cite{katuchova3}. One important property which distinguished tumor cells from normal cells is their ability to proliferate infinitely. This is the result of changes in cell death signaling pathways (apoptosis)\cite{katuchova4}. Angiogenesis is a next important factor for tumor growth. Growth of blood capillaries into the tumor is necessary for supply of nutrient and oxygen, and is induced by growth factors, such as Vascular Endothelial Growth Factor (VEGF). Angiogenesis is also required for metastases and tissue invasion of the tumor. Metastatic tumors are the cause of about 90\% of human cancer deaths \cite{katuchova5}. Spreading of metastases is possible through hematogenous and lymphogeneous pathways, which guide the metastases to other locations in the body, where they settle (intravasation). Both intravasation and the detachment from the original location (extravasation)  is characterized by changes in the extracellular matrix surrounding the tumor and its interactions with tumor cells. 

Molecular analysis indicates the importance of chemotactic motion in understanding of the outgrowth of tumor cells. As reported in \cite{katuchova15} importance of chemokines in tumor progression was obtained, e.g,  for breast cancer cells that typically metastasized in bone marrow, liver, lymph nodes and lung. These organs were found to secrete CXCL12, the ligand for the chemokine receptor CXCR4, which is enriched on breast cancer cells but not in normal breast epithelial cells \cite{katuchova15}.

Chemokines are now known to affect many aspects of tumor development such as angiogenesis and expression of cytokines, adhesion molecules, and proteases, and support of cancer cell migration. Thus chemotaxis plays an essential role in the successful outgrowth of tumors to the preferential organs. 

Another important factor that influences tumor growth is the immune system.  Physiologically normal immune system can effectively eliminates highly immunogenic tumor cells. However,  it can also happen that  tumor cells with a reduced immunogenicity can develop and further evase. Thus, the immune system has a selective function for tumor variants.  As time evolves  this selection leads to the growth of tumor cells that fail to  be controlled by the immune system.   In fact,  the interaction between a healthy immune system and tumor cells develops  in three phases:  the elimination, the equilibrium and the escape phase. The controlling role of the immune system, that determines whether and how tumors evolve in time is called immunoediting process \cite{katuchova17}.

Besides evasion of the immune system, chemotaxis towards CXCL12 is a key process also in the invasion of the Extracellular Matrix (ECM), see e.g. \cite{Condeelis.2011}. In more details, organs with high levels of specific chemokines can direct tumor cells that express the corresponding receptor, to their site; a result of chemotactic response and ECM invasion. This is the case for the pairs CXCR4–CXCL12 in bone metastasis of breast and prostate tumors.

Furthermore, the so-called urokinase plasminogen system plays an important role in cancer progression and metastasis. The proteins of this network on one hand help the tumor cell to remodel the extracellular matrix, so that it can detach from the original site and re-attach in another location. In addition, components of the network act as chemokines in order to guide the direction of tumor migration in this process. In the following, we describe the most important aspects of the corresponding mathematical model as introduced in \cite{chaplain2005mathematical}. 

\paragraph{The urokinase plasminogen activation system}~\\

Migration of (cancer) cells is a regulated process which involves de-attachment and attachment of the cell to the cellular matrix. The urokinase- plasminogen activation system is involved in several ways in this complex mechanism. The central role in this system play the urokinase-plasminogen activator (uPA) and its receptor on the cancer cell surface (uPAR). Although uPA is a protease and converts the protease plasminogen into its active form plasmin, also non-catalytical function is involved in the regulation of cancer cell migration. 

The protease uPA is secreted by the cancer cell in an inactive form (pro-uPA). This pro-form binds to uPAR and then can be activated by the protease plasmin. Receptor bound uPA has several functions: a) it can activate in turn plasmin by cleaving its pro-form plasminogen b) it enhances the affinity of uPAR to vitronectin \cite{Wei.1994}  and integrins c) uPA serves as chemotactic molecule and this action requires binding to uPAR. Activation of plasmin by uPA can also occur in solution, but is much enhanced if both enzymes are membrane bound. 

Vitronectin is a component of the extracellular matrix (ECM) and responsible for the attachment of cells to ECM. Integrins are transmembrane proteins which are responsible for signal transduction from the outside to the inside of the cell. Vitronectin and uPA/uPAR binds to integrins and by this are involved in cell signal events.   

All three types of function are regulated by the inhibitor PAI-1 (plasminogen activator inhibitor 1). This protein binds to uPA in the soluble and in the membrane-bound form, inhibiting its proteolytic function. Furthermore, it binds also to vitronectin \cite{Seiffert.1994}, and by this inhibits binding of this cell-adhesion molecule to uPA/uPAR, disrupting cell-ECM-contacts and also binding to integrins, interfering with cell-signalling. Furthermore, the complex uPAR/uPA/PAI-1 is removed from the cell surface by endocytosis, triggering further signaling pathways related to cancer migration. Besides cell signaling 
pathways related to cell migration, formation of the uPA/uPAR complexes was shown the increase proliferation of the cell. Chemotaxis induced by uPA is also inhibited by PAI-1. Furthermore, PAI-1 itself acts as chemotactic molecule and vitronectin has a similar role as haptotactic molecules, guiding the cell's movement on the ECM. The catalytic function of uPA, namely the activation of plasmin, is regarded as essential step in the cancer-cell's ability to remodel the ECM. 

The ECM is composed of a large number of biochemically and structurally diverse components, such as  proteins, proteoglycans, and glycoproteins. Formation of fibrillar structures by some of the proteins (eg. collagens, elastin) cause the particular mechano-elastic properties of the ECM. These structures are further connected by other proteins such as fibronectin or laminin.

Since intact ECM is a rather tight mesh, offering only small pores for the cell to move through, cleavage of ECM proteins by the cancer cells greatly enhances motility. However, which ECM-components can be cleaved by plasmin in vivo is not well investigated, see \cite{Deryugina.2012}, the best corrobated examples being laminin and fibronectin; however, MMPs which are activated by plasmin seem to have multiple roles in the regulation of tumor growth and progression (see  \cite{Kessenbrock.2010} for an overview).

\section{Mathematical model}\label{model}
In the literature already several mathematical models for various aspects of cancer invasion have been presented. The model for the cancer invasion that we investigate in this work is based on chemotactic/haptotactic movement, diffusion, enzyme interaction and mass conservation. It was first proposed \cite{chaplain2005mathematical, andasari2011mathematical}, later studied in \cite{andasari2011mathematical,gerisch2006robust,Gerisch.2008, kolbe2013master, Szymanska.2009}, and will be shortly presented in this section.

In this model, the ECM is represented by the component vitronectin ($v$). Furthermore, uPAR is not modeled as a separate entity, but is included via the cell-density. The other components (uPA, PAI-1, plasmin) are included explicitly, denoted by $u$, $p$, $m$. The differential equations for the different components are described below.
\begin{description}
\item[\textnormal{\textit{Cancer cells.}}]
	The spatio- temporal behavior of cancer cells $c$ is assumed to be determined by

	(i) random motion, modeled as diffusion; chemotaxis due to the gradients of (ii) uPA, and (iii) PAI-1; (iv) haptotaxis due to gradients of the ECM, chemo- and haptotaxis are modeled according to Keller and Segel \cite{keller1971model}; (v) proliferation of the cells themselves, which is assumed to be restricted by the cell number; (vi) increase of proliferation due to uPA/uPAR compounds.

	Consequently, the equation describing the cancer cell dynamics itself reads
	\begin{equation}\label{cancer}
		\partial_t c = \underbrace{D_c \triangle c}_{\text{(i)}} - \operatorname{div}(\underbrace{\chi_u c \nabla u}_{\text{(ii)}}  + \underbrace{\chi_p c \nabla p}_{\text{(iii)}} + \underbrace{\chi_v \nabla v}_{\text{(iv)}}) + \underbrace{\mu_1 c\left(1-\frac{c}{c_0}\right)}_{\text{(v)}}+ \underbrace{\phi_{1,3}~cu}_{\text{(vi)}} .
	\end{equation}

\item[\textnormal{\textit{Extracellular matrix.}}]
	As emphasized in the previous section, the extracellular matrix $v$ is a static object and thus no transport terms are needed in modeling it. ECM is represented here as vitronectin, for which the following dynamics are being considered

	(i) reconstruction of ECM is mathematically expressed with logistical growth; (ii) degradation of ECM is assumed to happen proportional to the product of the densities of plasmin and vitronectin, modeling in a rather global form any kind of degradation of ECM by plasmin, direct or indirect; (iii) release of $v$ from PAI-1/VN complexes due to competition of uPA for PAI-1; (iv) release of $v$ from PAI-1/VN complexes due to competition of uPA for PAI-1.

	The equation for the ECM dynamics finally reads
	\begin{equation}\label{VN}
		\partial_t v =\underbrace{\mu_2v\left(1-\frac{v}{v_0}\right)}_{\text{(i)}} -\underbrace{\delta vm}_{\text{(ii)}}  + \underbrace{\phi_{2,1}~up}_{\text{(iii)}} - \underbrace{\phi_{2,2}~vp}_{\text{(iv)}}.
	\end{equation}

\item[\textnormal{\textit{Urokinase plasmin activator.}}]
	For this component, $u$, the following processes are taken into account: 

	(i) uPA diffuses as chemical; (ii) uPA binding  to  cancer cell surface via uPAR receptors ; (iii) complex formation with the inhibitor PAI-1; (iv) secretion of uPA  by the cancer cells.

	Thus it is assumed that the following equation holds for the behavior of the urokinase dynamics
	\begin{equation}\label{uPA}
		\partial_t u = \underbrace{D_u \triangle u}_{\text{(i)}} - \underbrace{\phi_{3,3}~cu}_{\text{(ii)}} - \underbrace{\phi_{3,1}~pu}_{\text{(iii)}} +\underbrace{\alpha_3 c}_{\text{(iv)}}.
	\end{equation}

\item[\textnormal{\textit{Plasminogen activator inhibitor.}}]
	For this component, $p$, a similar set of terms is employed:

	(i) diffusion, similar to uPA; (ii) binding to uPA; (iii) binding to vitronectin; (iv) production by plasmin.

	Hence, the corresponding equation reads
	\begin{equation}\label{PAI-1}
		\partial_t p = \underbrace{D_p \triangle p}_{\text{(i)}} - \underbrace{\phi_{4,1}~pu}_{\text{(ii)}} - \underbrace{\phi_{4,2}~pv}_{\text{(iii)}} + \underbrace{\alpha_4 m}_{\text{(iv)}}.
	\end{equation}

\item[\textnormal{\textit{Plasmin.}}]
The ECM degenerating enzyme plasmin $m$ is controlled by the following dynamics

(i) chemical diffusion; (ii) it is activated by uPA/uPAR complexes; (iii) activation of plasmin is inhibited by PAI-1, which binds to uPA; iv) PAI-1/VN compounds result indirectly in production of plasmin, since bound PAI-1 does not inhibit plasmin formation anymore; (v) it is deactivated by plasmin inhibitors.

Thus the plasmin dynamics are modeled as follows
\begin{equation}
 \partial_t m = \underbrace{D_m \triangle m}_{\text{(i)}} + \underbrace{\phi_{5,3}~uc}_{\text{(ii)}} -\underbrace{\phi_{5,1}~pu}_{\text{(iii)}} + \underbrace{\phi_{5,2}~pv}_{\text{(iv)}}  -\underbrace{\alpha_5 m}_{\text{(v)}}.
\end{equation}
\end{description}

To formulate the model in non-dimensional variables, rescaling takes place using reference length $L=0.1$ cm, a reference diffusion coefficient $D=10^{-6}$ $\text{cm}^2\text{s}^{-1}$, a rescaled time parameter $t=L^2D^{-1}$, and reference densities $C$, $V$, $U$, $P$, $M$ of the cancer cells, vitronectin, uPA, PAI-1 and plasmin respectively.

As we only consider dimensionless variables, we keep the former notations for the densities of cancer cells, ECM, and proteins, and end up with the system
\begin{equation}\label{chaplolsystem}
\left\{
\begin{aligned}
\partial_t c &= D_c \Delta c &- \operatorname{div}(\chi_u c\nabla u + \chi_p c \nabla p + \chi_v c\nabla v) + \phi_{1,3}cu&+\mu_1c(1-c),\\
\partial_t v & = &- \delta vm + \phi_{2,1}up - \phi_{2,2}vp &+ \mu_2v(1-v),\\
\partial_t u &= D_u \Delta u &-\phi_{3,1}pu - \phi_{3,3}c u &+ \alpha_3 c,\\
\partial_t p &= D_p \Delta p &- \phi_{4,1}pu - \phi_{4,2}pv &+ \alpha_{4}m,\\
\partial_t m &= D_m \Delta m &- \phi_{5,1} pu + \phi_{5,2} pv + \phi_{53}uc &- \alpha_{5}m.
\end{aligned}
\right.
\end{equation}

To simplify the presentation in the following sections we introduce the short notations
\begin{equation}
	\w = \Vector{c\\ v\\ u\\ p\\ m},~D(\w)=\Vector{D_c \Delta c\\ 0\\ D_u \Delta u \\D_p \Delta p \\D_m \Delta m},
 ~A(\w) = \Vector{\operatorname{div}(\chi_u c\nabla u + \chi_p c \nabla p + \chi_v c\nabla v)\\0\\0\\0\\0},
\end{equation}
of the variables, diffusion, advection and reaction vectors
\begin{equation}\label{chaplolreaction}
  R(\w) = \begin{pmatrix} \phi_{13}cu+\mu_1c(1-c)\\- \delta vm + \phi_{21}up - \phi_{22}vp + \mu_2v(1-v)\\-\phi_{31}pu - \phi_{33}c u + \alpha_3 c\\
- \phi_{41}pu - \phi_{42}pv + \alpha_{4}m\\- \phi_{51} pu + \phi_{52} pv + \phi_{53}uc - \alpha_{5}m \end{pmatrix},
\end{equation}
respectively. We note that the vectors for advection and diffusion include derivatives of $\w$ as well which we have not included in our notation for brevity. With this notation the system \eqref{chaplolsystem} recasts into
\begin{equation}\label{rdt}
 \w_t + A(\w) = D(\w) + R(\w).
\end{equation}

The parameter set $\mathcal{P}$, cf. e.g. \cite{andasari2011mathematical}, which we also consider in this work is given by
\begin{equation}\label{params}
\left\{
\begin{array}{llll}
D_c = 3.5 \cdot 10^{-4},	& 	\chi_u = 3.05\cdot 10^{-2},		& 	\mu_1 = 0.25,	&	\alpha_3 = 0.215,\\
D_u = 2.5\cdot 10^{-3},		& 	\chi_p=3.75\cdot 10^{-2},		&	\mu_2=0.15,		&	\alpha_4 = 0.5,\\
D_p=3.5\cdot 10^{-3}, 		&	\chi_v = 2.85\cdot 10^{-2},		&	\delta=8.15,	&	\alpha_5=0.5, \\
D_m=4.91\cdot 10^{-3},		&	\phi_{13} = 0,					& 	\phi_{21}=0.75,	&	\phi_{22}= 0.55,\\	
\phi_{31}=0.75, 			&	\phi_{33}=0.3,					&	\phi_{41}=0.75, &	\phi_{42}=0.55,\\
\phi_{51}= 0,				&	\phi_{52}=0.11,					&	\phi_{53}=0.75,
\end{array}
\right.
\end{equation}

These parameters were estimated by fitting the numerical results to in-vitro experiments, see \cite{chaplain2005mathematical, lolas2003phd} for details. Although we do not address the parameters in detail, we note that chemo-, haptotactical sensitivities $\chi_p,~\chi_u$ and $\chi_v$ are approximately a factor of hundred times higher than the coefficient $D_c$. Thus we expect the motion of the cancer cells to be dominated by taxis. We further note that the parameter $\delta=8.15$ is much higher in comparison with the  other parameters used, due to the crucial role of tissue degeneration by plasmin.

\section{Analytical properties of a chemotaxis-haptotaxis model with logistic source}\label{analysis}
Solutions of the system \eqref{chaplolsystem} feature heterogeneous spatio-temporal dynamics in the form of emerging, merging and traveling concentrations. This was observed numerically and examined by steady state analysis in \cite{andasari2011mathematical}. In the latter it is commented that these dynamics are the results of the destabilization of a single steady state of $\w_t = R(\w)$ by advection.

To give more details, let $\hat{\w}\in \mathbb{R}^n$ be a positive steady state of $\w_t = R(\w)$. A small pertubation $\w(t,x)=\hat{\w}+\eps \tilde \w(t,x)$ evolves according to
\[\tilde \w_t = J_R(\hat{\w})\tilde \w + J_T(\hat \w) \Delta \tilde \w + \mathcal{O}(\eps^2),\]
where $J_R(\w)$ and $J_T(\w)$ are the Jacobians of the reaction $R$ and the general transport operator $T$ such that $\divergence T(\nabla \w)=D(\mathbf w)-A(\w)$. We denote by $J_R(\w)=D_{\w}R(\w)$,  $J_T(\w)=D_{\nabla \w}T(\nabla \w)$, respectively.


Assuming periodic boundary conditions on $(-M,M)^d$ we write the pertubation $\tilde \w$ in a Fourier series representation,
\[\tilde \w(t,x) = \sum_{\bnu \in \mathbb{Z}^d} \a_{\bnu}(t) \exp(i \bnu \pi M^{-1}\cdot x), \quad
\a_{\bnu} :~ [0,T]\rightarrow \mathbb{R}^n.\]
The evolution of the coefficient functions is determined by the following system of ordinary differential equations
\[\a_{\bnu}'(t) = \left( J_R(\hat \w)- k J_T(\hat \w)\right)\a_{\bnu}(t),\quad \a_{\bnu}(0) = \a_{\bnu,0}, \quad \bnu \in \mathbb{Z}^d,\] where \[ \tilde \w(0,x) = \sum_{\bnu \in \mathbb{Z}^d} \a_{\bnu,0} \exp(i \bnu \pi M^{-1} \cdot x), \quad k = \|\bnu \pi M^{-1}\|_2^2.\]
Thus we can see, that pertubations due to wave number $k$ grow, if
\[\lambda_{k}(\hat \w) = \max\left\{\operatorname{Real}\left\{ \operatorname{spec}\left( J_R(\hat \w)- k J_T(\hat \w) \right)\right\}\right\} > 0;\]
otherwise they are damped. In the case of parameter set $\mathcal{P}$, there is a range of numbers $k$ with $\lambda_{k} > 0$, which vanishes if chemotaxis is neglected (that is if $\chi_u=\chi_p=\chi_v=0$).

In order to understand the dynamics of \eqref{chaplolsystem}, such as the merging and emerging phenomena, and to develop an efficient, problem suited, adaptive numerical scheme, let us first consider the following chemotaxis-haptotaxis model with the logistic source term
\begin{equation}\label{small_system}
 \left\{
 \begin{aligned}
\partial_t c &= D_c \Delta c -\nabla \cdot(\chi c \nabla u) + \mu c(1-c), \\
\partial_t u &= D_u \Delta u + \alpha c -\beta u.
 \end{aligned}
 \right.
\end{equation}
The system \eqref{small_system} provides similar dynamics as \eqref{chaplolsystem}. This has been examined in numerical experiments in \cite{painter2011spatio}. In what follows, we demonstrate that the model \eqref{small_system} also obtains increasing modes of pertubations due to chemotaxis. If we neglect the advection terms in \eqref{small_system} we get a system of ordinary differential equations for $\w = (c, u)^T$ with steady state $\hat \w = (1, \frac{\alpha}{\beta})^T$, in which case, the Jacobians of reaction and generalized transport are given by
\[J_R(\w) = \begin{pmatrix} \mu(1-2c) & 0 \\
\alpha & -\beta \end{pmatrix},
\quad
J_T(\w) = \begin{pmatrix} D_c & -\chi c \\
0 & D_u \end{pmatrix}.\]
Choosing the parameters
\[D_c = 5.25 \cdot 10^{-3},\quad D_u=2.5\cdot 10^{-3},\quad
\chi = 4 \cdot 10^{-2}, \quad
\mu = 0.1,\quad \alpha = 0.115, \quad\beta = 0.4,\]
we obtain a range of wave numbers $k$ with a positive value of $\lambda_{k}(\hat \w)$, as it is shown in Figure \ref{disp_rel}. This implies that small pertubations due to chemotaxis can increase in time. Analogous behaviour has been obtained for system \eqref{chaplolsystem} in \cite{andasari2011mathematical} and for \eqref{small_system} with $\alpha = \beta = 1$ in \cite{painter2011spatio}.

\begin{figure}
 \centering
 \includegraphics[width=1\textwidth]{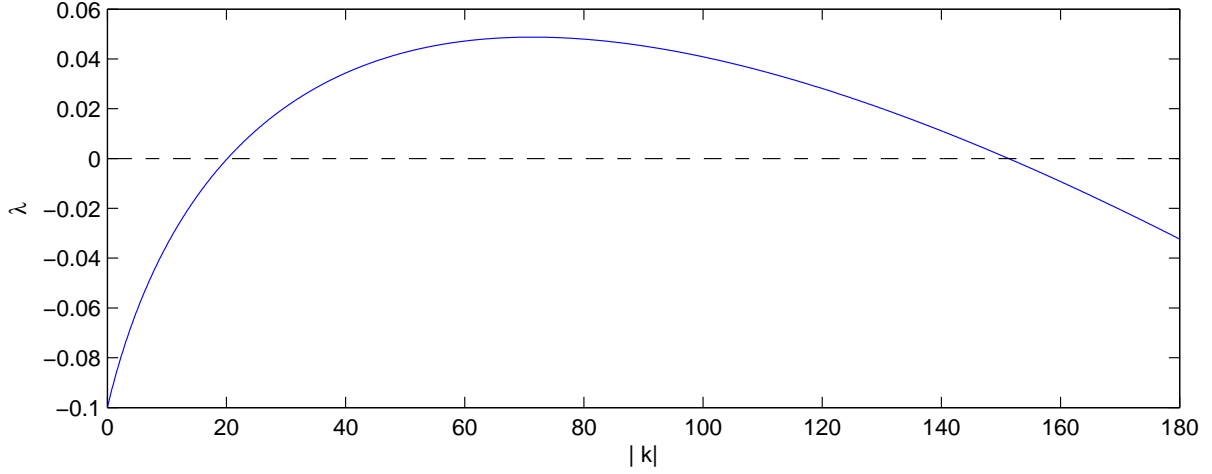}
 \caption{A plot of $\lambda_{k}(\hat \w)$ against $k$. The amplification factor is positive for a wide range of modes. } \label{disp_rel}
\end{figure}

In \cite{chertock2012chemotaxis,kurganovnumerical} the authors studied the so-called saturated chemotaxis model and were able to show the global existence and uniqueness of a classical solution for the chemotaxis model of one and two species, respectively. In what follows we analyze the chemotaxis-haptotaxis model with a logistic source term \eqref{small_system}, modified to include saturated chemotaxis flux, i.e.

\begin{equation}\label{small_system_reg}
 \left\{
 \begin{aligned}
\partial_t c &= D_c \Delta c -\nabla \cdot( c\, Q(\chi \nabla u)) + \mu c(1-c), \\
\partial_t u &= D_u \Delta u + \alpha c -\beta u,
 \end{aligned}
 \right.
\end{equation}
where $Q$ is given by
\begin{equation}\label{Q_func}
 Q(\chi \nabla u) =
 \begin{dcases}
  \chi \nabla u, 			&\text{if }\chi|\nabla u| \leq S, \\
  \left( \frac{\chi |\nabla u| -S}{\sqrt{1+(\chi |\nabla u| -S)^2}}+S\right)\frac{\nabla u}{|\nabla u|}, &\text{otherwise,}
 \end{dcases}
\end{equation}
for a positive constant $S$. The function $Q$ limits the flux by $\|Q(\chi \nabla u)\|<S+1=C$.
\begin{theorem}\label{Linfty_estimates}
	We consider \eqref{small_system_reg} on a compact set $\Omega\subset \mathbb{R}^d$ with a Lipschitz continuous boundary $\partial \Omega$ with the outer normal $\mathbf n$ and the boundary conditions
	\begin{equation}
		\frac{\partial c}{\partial \mathbf n} = \frac{\partial u}{\partial \mathbf n} = 0, \quad x\in\partial\Omega,~t>0.
	\end{equation}
	Let $(c(x,t),~u(x,t))$ be a positive classical solution with bounded non-negative initial data. Then the following estimates hold,
	\begin{align}
		c(x,t) &\leq   \tilde{\tilde C} \max \left\{ \| c_o\|_{L^\infty(\Omega)}, ~ \|c_0\|_{L^1(\Omega)}+ \frac{T\,\mu |\Omega|}{4}\right\},\\
		u(x,t) &\leq \|u_0\|_{L^\infty(\Omega)}+\frac{\tilde{\tilde C} \alpha}{\beta} \max \left\{ \| c_o\|_{L^\infty(\Omega)}, ~ \|c_0\|_{L^1(\Omega)}+ \frac{T\,\mu |\Omega|}{4}\right\},
	\end{align}
	for all $x\in\bar \Omega$ and $t\in[0,T]$, where
	\begin{equation}
		\tilde{\tilde C} = \tilde C \left( 1 + \frac{2D_c \mu}{C^2}\right)^{2} \left( 1 + \frac{\sqrt{C^2+2\mu D_c}}{D_c}\right)^{2d},
	\end{equation}
	and $\tilde C$ depends on $\Omega$ only.
\end{theorem}
\begin{proof}
We follow \cite{kurganovnumerical} and start by multiplying the first equation of \eqref{small_system_reg} by $c^{s-1}$ for $s\geq2$. Integration by parts, the chain rule, the bound on $Q$ as well as Young's inequality yield,
 \begin{align*}
\frac 1s \frac{d}{dt} \int_\Omega c^s\,dx &= -D_c \int_\Omega \nabla c \cdot \nabla(c^{s-1})\,dx +  \int_\Omega c Q(\chi \nabla u) \cdot \nabla(c^{s-1})\,dx +\mu \int_\Omega c^s(1-c)\, dx \\
 &\leq - \frac{4D_c(s-1)}{s^2}\int_\Omega|\nabla(c^{\frac s 2})|^2 \,dx  + \frac{2C(s-1)}{s}\int_\Omega c^{\frac s2} |\nabla(c^{\frac s 2})| \,dx + \mu \int_\Omega c^s\, dx \\
 & \leq - \frac{4D_c(s-1)}{s^2} \|\nabla(c^{\frac s 2}) \|_{L^2(\Omega)}^2 + \frac{2C(1-s)}{s} \int_\Omega \frac 12 \left(\frac{2D_c}{Cs} |\nabla(c^{\frac s 2})| + \frac{Cs}{2D_c} c^s\right)\, dx + \mu \int_\Omega c^s\, dx \\
 &\leq  - \frac{2D_c(s-1)}{s^2} \|\nabla(c^{\frac s 2}) \|_{L^2(\Omega)}^2 + \frac{C^2(s-1)+2D_c\mu}{2D_c}\int_\Omega c^s\, dx.
 \end{align*}
The last integral is estimated using the interpolation inequality,
\begin{equation}
 \|\omega\|^2_{L^2(\Omega)} \leq \eps \|\nabla \omega\|^2_{L^2(\Omega)} + K(1+\eps^{-\frac d2}) \| \omega\|^2_{L^1(\Omega)},
\end{equation}
for a constant $K$ depending only on the domain $\Omega$ and a chosen $\eps$ \cite{ladyzhenskaeiia1968linear}. We choose $\eps$ such that
\[ \frac{C^2(s-1)+2D_c\mu}{2D_c} = \frac{2D_c(s-1)}{s^2\eps}- \frac{C^2(s-1)+2D_c\mu}{2D_c} \Rightarrow \eps = \frac{2D_c^2(s-1)}{C^2s^2(s-1)+2D_c\mu s^2}>0,\] and thus get
\begin{align*}
 \frac{C^2(s-1)+2D_c\mu}{2D_c}\int_\Omega c^s\, dx &= \left( \frac{2D_c(s-1)}{s^2\eps}- \frac{C^2(s-1)+2D_c\mu}{2D_c}\right) \int_\Omega c^s\, dx \\
 &\leq \frac{2D_c(s-1)}{s^2} \|\nabla(c^{\frac s 2}) \|_{L^2(\Omega)}^2 + \frac{2D_c(s-1)K(1+\eps^{-\frac d2})}{s^2\eps} \| c^{\frac s 2}\|^2_{L^1(\Omega)} \\ &- \frac{C^2(s-1)+2D_c\mu}{2D_c} \int_\Omega c^s\, dx.
\end{align*}
Hence the dependence on $\nabla(c^{\frac s 2})$ in the above estimates vanishes and using $\eps^{-\frac 12}\leq s\sqrt{C^2+2\mu D_c}/D_c$, we get
\[ \frac{d}{dt}\int_\Omega c^s\,dx   + \frac{sC^2(s-1)+2sD_c\mu}{2D_c} \int_\Omega c^s\, dx\leq \frac{Ks\left(1+\left( \frac{s\sqrt{C^2+2\mu D_c}}{D_c}\right)^d \right)(C^2(s-1)+2D_c\mu)}{D_c}
\left( \int_\Omega c^{\frac s2} \, dx\right)^2.\]

Next, we multiply by the integrating factor $e^{\kappa t},~\kappa = (sC^2(s-1)+2sD_c\mu)/2D_c$, and obtain after integrating over $[0,t],~0<t\leq T$ and eliminating the integrating factor
\begin{equation}
 \int_\Omega c^s(x,t)\,dx \leq  \int_\Omega c_0^s \, dx + 2K\left( 1 + \frac{2D_c \mu}{C^2(s-1)}\right)\left( 1 + \frac{s\sqrt{C^2+2\mu D_c}}{D_c}\right)^d \sup_{0\leq t \leq T} \left( \int_\Omega c^{\frac s2} \, dx\right)^2
\end{equation}
Now, the function
\[M(s) = \max\left\{\|c_0\|_{L^\infty(\Omega)},~\sup_{0\leq t \leq T}\left( \int_\Omega c^{\frac s2} \, dx\right)^{\frac 1s} \right\}\]
satisfies the inequality
\[M(s)\leq\left(\tilde K \left( 1 + \frac{2D_c \mu}{C^2(s-1)}\right)\left( 1 + \frac{s\sqrt{C^2+2\mu D_c}}{D_c}\right)^d \right)^{\frac 1s}M(s/2).\]
Choosing the sequence $s=2^k,~k\in\mathbb{N}$, and dissolving the recursion by estimating the value of the monotonously increasing infinite product,
\[\prod_{k=1}^\infty \left( 1 + 2^k C \right)^{\frac{d}{2^k}}\leq (2 + 2C)^{2d},\]we get
\[M(2^k)\leq  \tilde{C}\left( 1 + \frac{2D_c \mu}{C^2}\right)^{2} \left( 1 + \frac{\sqrt{C^2+2\mu D_c}}{D_c}\right)^{2d}M(1),\]
where $\tilde{C}=2 \max\{\tilde K,~\tilde K ^2\}$ depends on $\Omega$ only.
Taking the limit $k\rightarrow \infty$, we end up with
\begin{equation}\label{linftyc}
 \|c(\cdot,t)\|_{L^{\infty}(\Omega)}\leq \tilde C \left( 1 + \frac{2D_c \mu}{C^2}\right)^{2} \left( 1 + \frac{\sqrt{C^2+2\mu D_c}}{D_c}\right)^{2d}M(1) = \tilde{\tilde C} M(1).
\end{equation}
We further consider the mass of the component c by integrating the the first equation of \eqref{small_system_reg} and obtain applying Gauss's theorem and using the non-negativity of $c$,
$$\frac{d}{dt}\int_\Omega c\,dx =\mu\int_\Omega c(1-c) \,dx\leq \mu \int_\Omega \frac 14 \,dx = \frac{\mu |\Omega|}{4},$$
since $c(1-c)\leq 1/4$ for any $c\geq 0$. Therefore we have
$$M(1) = \max \left\{ \| c_o\|_{L^\infty(\Omega)}, ~ \|c_0\|_{L^1(\Omega)}+ \frac{T\,\mu |\Omega|}{4}\right\}.$$
Together with \eqref{linftyc} the estimation for $c$ follows.

Because of the maximum principle of heat equation, the solution of the following initial value problem,
$$\left\{
 \begin{aligned}
  \frac{d\omega}{dt} &= -\beta \omega + \alpha \tilde{\tilde C} M(1),\\
  \omega(0) &= \|u_0\|_{L^\infty(\Omega)},
 \end{aligned}
\right.$$
is an upper bound for $u$. Hence we can estimate
\begin{align*}
 0\leq u(x,t) \leq \omega(t) = e^{-\beta t}\|u_0\|_{L^\infty(\Omega)} + (1-e^{-\beta t})\frac{\alpha \tilde{\tilde C} }{\beta} M(1) \leq e^{-\beta t}\|u_0\|_{L^\infty(\Omega)} + \frac{\alpha \tilde{\tilde C}}{\beta} M(1),
\end{align*}
which proves the $L^\infty(\Omega)$ bound of the density $u$.
\end{proof}
\begin{remark}
Existence, uniqueness and non-negativity of a classical solution for system \eqref{small_system} with homogeneous Neumann boundary conditions and non-negative initial data $c_0,~u_0\geq 0$ was shown in \cite{wrzosek2004global, tao2009global}. Hence, for a sufficiently large $S$ in \eqref{Q_func} and non-negative initial conditions we can assume that the conditions of Theorem \ref{Linfty_estimates} hold.
\end{remark}

\begin{remark}
	Additionally, in \cite{Czochra.2010}, the authors studied another simplified cancer invasion model that includes also the degradation of the ECM. They employ nonlinear change of variables and comparison principles to prove existence, uniqueness, positivity, and boundedness of the solutions.
\end{remark}

\section{Numerical Methods}\label{methods}
We perform numerical simulations on a computational domain $\Omega$, which is either an interval in 1D or a rectangular domain in 2D subdivided into a finite number of computational non overlapping cells
\begin{align*}\Omega = \bigcup_{i=1}^N C_i .\end{align*} 
In the one-dimensional case on an interval domain $\Omega=(a,b)$, the cell interfaces are given by
$$ a = x_{1/2},\quad x_{i+1/2}= x_{i-1/2}+h_{i},\quad x_N = b,$$
for given cell sizes $h_i>0$ satisfying $\sum_{i=1}^N h_i = b-a$. Thus cell centers and cells are defined by
$$ x_i = \frac{x_{i+1/2}-x_{i-1/2}}{2},\quad C_i=\left\{ x_i + \lambda h_i,~\lambda \in \left[-\frac12, \frac12\right) \right\} \quad  i=1,\dots,N.$$
We will be employing mesh adaptation and although our grid refinement techniques can be adapted to general two dimensional meshes, we only consider rectangular domains $\Omega=(a,b)\times (a,b)$ with uniform quadrilateral cells in this work.
Consequently we employ constant grid sizes
$$h=(h^{(1)},h^{(2)})^T,~ h^{(1)}=\frac{b-a}{L},~h^{(2)}=\frac{b-a}{M},\quad N = L\, M, \quad L,M\in\mathbb{N}.$$
This way, cell centers are given by
\begin{align*}
	x_{1,1} &= \mathbf{e}_1(a+\frac{h^{(1)}}{2}) + \mathbf{e}_2(a+\frac{h^{(2)}}{2}),\\
	x_{i,j} &= x_{1,1} + \mathbf{e}_1(i-1)h^{(1)} + \mathbf{e}_2(j-1)h^{(2)},\quad &i=1,\dots,L,~j=1,\dots,M,
\end{align*}
where $\mathbf e_1, \mathbf e_2$ are the unit vectors in the spatial directions $x_1$ and $x_2$, respectively. Computational cells for $i=1,\dots,L,~j=1,\dots,M$ are given by
$$ C_{i,j}=\left\{ x_{i,j} + \Matrix{\lambda_1\, h^{(1)} \\\lambda_2\, h^{(2)}},~\lambda_1,\lambda_2 \in \left[-\frac12, \frac12\right) \right\}.$$
Moreover, we introduce a single index notation, for cells and cell centers using lexicographical indexing, i.e.
$$C_{i,j} \longrightarrow C_{i+(j-1)L},\quad x_{i,j} \longrightarrow x_{i+(j-1)L}, \quad i=1,\dots,L,\quad j=1,\dots,M,$$
and inversely,
$$C_k \longrightarrow C_{k-[\frac{k-1}{L}] L ,[\frac{k-1}{L}] +1},\quad x_{k} \longrightarrow x_{k-[\frac{k-1}{L}]L ,[\frac{k-1}{L}] +1}, \quad k=1,\dots,LM,$$
where $[~]$ is the Gauss floor function.

Let moreover $C_{k\pm \mathbf e_l}$ denote the neighbouring cell of $C_k$ in positive(negative) $\mathbf e_l$ direction ($l=1,2$), i.e., for $k=1,\dots ML$ we define

\begin{align*}
	C_{k \pm \mathbf e_1}= & C_{k-[\frac{k-1}{L}] L \pm  1,[\frac{k-1}{L}] +1}, \quad \text{where }k\neq 0,1\, \mod L,\text{ respectively }(\pm),\\
	C_{k \pm \mathbf e_2}= & C_{k-[\frac{k-1}{L}] L ,[\frac{k-1}{L}] +1 \pm 1}, \quad \text{for }k\leq  L(M-1),\ k\geq  L+1, \text{respectively }(\pm).\\
\end{align*}
									
\paragraph{Space discretization.} We discretize the system \eqref{chaplolsystem} with a finite volume method, and approximate a solution with piecewise constant functions
\begin{equation}\label{approx}
	\w_i(t)\approx \frac{1}{\operatorname{vol}(C_i)}\int_{C_i} \w(x,t)~dx,
\end{equation}
on every cell. In the following we present discretization of the advection, reaction, and diffusion operators,
\begin{align*}
	\mathcal{D}_i(\w_h(t)) &\approx \frac{1}{\operatorname{vol}(C_i)}\int_{C_i} D(\w(x,t))~dx,
 							&\mathcal{A}_i(\w_h(t)) &\approx \frac{-1}{\operatorname{vol}(C_i)}\int_{C_i} A(\w(x,t))~dx, \\
	\mathcal{R}_i(\w_h(t)) &\approx \frac{1}{\operatorname{vol}(C_i)}\int_{C_i} R(\w(x,t))~dx, &&
\end{align*}
where $\w_h(\cdot)=\{\w_i(\cdot)\}_{i=1}^N.$
We discretize the reaction term by evaluating the reaction operator
$$\mathcal{R}_i(\w_h(t)) = R(\w_i(t)).$$
Concerning diffusion in dimension 1, we use second order three-point central differences when the grid is uniform and second order five point central differences when the grid is non-uniform with $h_i=|C_i|$. In the latter case the discretization of the diffusion operator reads
\begin{equation}\label{eq:dscrDiff}
	\mathcal{D}_i(\w_h(t)) = D \left( \alpha_i^{(-2)} \w_{i-2}(t) +\alpha_i^{(-1)} \w_{i-1}(t) + \alpha_i^{(0)} \w_i(t) + \alpha_i^{(+1)} \w_{i+1}(t) + \alpha_i^{(+2)} \w_{i+2}(t) \right),
\end{equation}
where $D$ is a diagonal matrix with the vector $\Vector{ D_c & 0 &D_u &D_p &D_m}$ on the diagonal, and the coefficients $\alpha_i^{(-2)}, \dots, \alpha_i^{(+2)}$, $\sigma_i$, are chosen such that we get a second order approximation of the second derivative, i.e.
\begin{align*}
\alpha_i^{(-2)} 	=& - 8\frac{(h_{i-1}-h_{i+1})}{(h_{i-2}+2h_{i-1}+2h_i + 2 h_{i+1} + h_{i+2}) \sigma_i},\\
\alpha^{(-1)}_i 	=& 8\frac{h_{i-1}(4h_{i-1}+4h_{i-2}+2h_i-4h_{i+1}-2h_{i+2})+3h_{i+1}^2}{(h_i + h_{i-1})(h_{i-1}+2h_i+h_{i+1})\sigma_i}\\
					 &+ 8\frac{h_{i+2}^2+4h_{i+1}h_{i+2}+h_i h_{i+2}+h_{i-2}(h_{i-2}-2h_{i+1}-h_{i+2}+h_i)}{(h_i + h_{i-1})(h_{i-1}+2h_i+h_{i+1})\sigma_i},\\
\alpha^{(+1)}_i 	=& 8\frac{h_{i+1}(4h_{i+1}+4h_{i+2}+2h_i-4h_{i-1}-2h_{i-2})+3h_{i-1}^2}{(h_i + h_{i+1})(h_{i-1}+2h_i+h_{i+1})\sigma_i}\\
					 &+ 8\frac{h_{i-2}^2+4h_{i-1}h_{i-2}+h_i h_{i-2}+h_{i+2}(h_{i+2}-2h_{i-1}-h_{i-2}+h_i)}{(h_i + h_{i+1})(h_{i-1}+2h_i+h_{i+1})\sigma_i},\\
\alpha_i^{(0)}		=&-(\alpha^{(-1)}_i+\alpha^{(+1)}_i), \\
\alpha_i^{(+2)}		=&-\alpha_i^{(-2)}, \\
\sigma_i 			=& h_{i-2}^2 + h_{i+2}^2 + 2 (h_{i-1}^2 +h_{i+1}^2) + 3 (h_{i-1}h_{i-2} + h_{i+1}h_{i+2})\\
					 &+ h_i( h_{i+1} + h_{i-1} + h_{i+2} + h_{i-2})- h_{i-2}(h_{i+1} + h_{i+2}) - h_{i-1}(h_{i+1} + h_{i+2}),
\end{align*}
The description \eqref{eq:dscrDiff} reduces to the common three-point central differences if the grid is uniform around $C_i$.


In this work we discretize the advection terms following the guidelines of \cite{kurganovnumerical}. The discrete advection operator in the conservative formulation reads
\begin{equation}
	\mathcal{A}_i(\w_h(t)) = -\sum_{j=1}^d \frac{1}{h_i}\Vector{\mathcal{H}_{i+\mathbf e_j/2}(\w_h(t))-\mathcal{H}_{i-\mathbf e_j/2}(\w_h(t))
	& 0 & 0 & 0 & 0}^T,
\end{equation}
where $d\in\{1,~2\}$ is the dimension of the domain. The numerical fluxes $\mathcal{H}_{i+\mathbf e_j/2}$, cf. \eqref{numFlux}, are used to approximate the chemotaxis fluxes between the cells $C_i$ and $C_{i+\mathbf e_j}$. They are given by products of approximated characteristic velocities $\mathcal{P}_{i+\e_j/2}$ and suitable approximations of gradients of $u,~v$ and $p$, denoted here by $s_{i}^{(j)}$. To describe $\mathcal{H}_{i+\mathbf e_j/2}$ we need to define $\mathcal{P}_{i+\mathbf e_j/2}$ and $s_{i}^{(j)}$. For the approximation of characteristic velocities on the cell interfaces we calculate approximations of the form
$$\mathcal{P}_{i+\mathbf e_j/2}(\w_h(t))  = \chi_u L_{i+\mathbf e_j/2}(u_h(t)) + \chi_v L_{i+\mathbf e_j/2}(v_h(t)) + \chi_p L_{i+\mathbf e_j/2}(p_h(t)),$$
where $L^\text{diff}_{i+\mathbf e_j/2}$ represents central difference approximations of the first derivative. Since, 2nd order approximations cannot be obtained by a three point stencil on non-uniform grids, we apply a four point finite difference approximation centered around the interface, i.e.
$$L_{i+1/2}(u_h) = \beta^{(-1.5)}_{i+1/2} u_{i-1} + \beta^{(-0.5)}_{i+1/2} u_{i} + \beta^{(+0.5)}_{i+1/2} u_{i+1}+ \beta^{(+ 1.5)}_{i+1/2} u_{i+2},$$
where the coefficients are chosen such that we have a third order accurate approximation of the first derivative,
\begin{align*}
 \beta^{(-1.5)}_{i+1/2} &= \frac{h_{i+1}(6 h_{i}-4h_{i+1}-2h_{i+2}) + 2h_{i} h_{i+2}}{(h_i+h_{i-1})(h_{i-1}+2h_i + h_{i+1})(h_{i-1}+2 h_i + 2h_{i+1} + h_{i+2})}, & &\\
 \beta^{(-0.5)}_{i+1/2} &= - \frac{h_{i+1}(12h_i + 6 h_{i-1} -2 h_{i+2}-4h_{i+1})+h_{i+2}(2h_{i-1}+4h_i)}{(h_{i-1} + h_i)(h_i+h_{i+1})(h_i + 2 h_{i+1}+ h_{i+2})}, & &\\
 \beta^{(+0.5)}_{i+1/2}&= \frac{h_{i}(12h_{i+1} + 6 h_{i+2} -2 h_{i-1}-4h_{i})+h_{i-1}(2h_{i+2}+4h_{i+1})}{(h_{i+1} + h_{i+2})(h_i+h_{i+1})(h_{i-1} + 2 h_{i}+ h_{i+1})}, & \\
 \beta^{(+1.5)}_{i+1/2} &= - \frac{h_i(6 h_{i+1} - 4h_i- 2h_{i-1}) + 2h_{i+1}h_{i-1}}{(h_{i+1}+h_{i+2})(h_{i}+2h_{i+1} + h_{i+2})(h_{i-1}+2 h_i + 2h_{i+1} + h_{i+2})}. &&
\end{align*}

In order to approximate the gradient of $c$, we apply the MC limiter \cite{van1977towards}. On the uniform grid it reads
$$s_{i}^{(j)}(c_h)=\minmod \left( 2\frac{c_i-c_{i-\mathbf e_j}}{h}, ~ \frac{c_{i+\mathbf e_j}-c_{i-\mathbf e_j}}{2h}, ~ 2\frac{c_{i+\mathbf e_j}-c_i}{h} \right).$$

In the case of non-unform grids in one dimension we obtain
$$s_{i}(c_h) = \minmod \left( 4 \frac{c_i-c_{i-1}}{\kappa_{i-1}},~   \frac{-2 \kappa_i c_{i-1}}{\kappa_{i-1}(\kappa_{i-1}+\kappa_i)}  +
\frac{2(\kappa_i^2-\kappa_{i-1}^2)c_i}{\kappa_{i-1}\kappa_i(\kappa_{i-1}+\kappa_i)}  +
\frac{2\kappa_{i-1}c_{i+1}}{\kappa_i(\kappa_{i-1}+\kappa_i)} ,~4\frac{c_{i+1}-c_i}{\kappa_i} \right),$$
where $\kappa_i = h_i + h_{i+1}$. The minmod operator is given by
\begin{equation}\label{minmod}
	\operatorname{minmod}(v_1,\dots,v_n)=
	\begin{dcases}
		\max\{v_1,\dots,v_n\},	&\textnormal{if}~ v_k<0,~k=1,\dots,n,\\
		\min\{v_1,\dots,v_n\},	&\textnormal{if}~ v_k>0,~k=1,\dots,n, \\
		0,						&\textnormal{otherwise}.
	\end{dcases}
\end{equation}
The numerical fluxes are based on the upwinding approach and are given as follows
\begin{equation} \label{numFlux}
	\mathcal{H}_{i+\e_j/2}(\w_h)=
	\begin{dcases}
		\mathcal{P}_{i+\e_j/2}(\w_h) \left( c_i +\frac{h_i^{(j)}}{2}s_{i}^{(j)}(c_h) \right),			&\textnormal{if }\mathcal{P}_{i+\e_j/2}(c_h)\geq 0, \\
		\mathcal{P}_{i+\e_j/2}(\w_h) \left( c_{i+\e_j} - \frac{h_{i+1}^{(j)}}{2}s_{i+\e_j}^{(j)}(c_h) \right),	&\textnormal{if }\mathcal{P}_{i+\e_j/2}(w_h)< 0.
	\end{dcases}
\end{equation}
After space discretization, we end up with the following system of ordinary differential equations
\begin{equation}\label{approx_ODE}
	\partial_t \w_h -\mathcal{A}(\w_h)= \mathcal{R}(\w_h) + \mathcal{D}(\w_h).
\end{equation}

\paragraph{Time discretization.}
The numerical approximations of the solution of \eqref{approx_ODE} at discrete points in time $t_n$ wil be denoted 
$\w_h^n$. 

The discretization of the taxis-terms dictates a limit on the time step $\tau_n=t_{n+1}-t_n$ according to the CFL condition
\begin{equation}\label{CFL}
	\max_{i, j}~ \tau_n \frac{\mathcal{P}_{i+\mathbf e_j}(\w_h^n}{h_i} = CFL \leq 1.
\end{equation}
We have compared several numerical methods to approximate \eqref{approx_ODE}. They are shortly described in the following part.

\begin{description}
\item[\textnormal{\textit{EXPLICIT:}}]
	The first order forward Euler time integration
	$$\w_h^{n+1} = \w_h^n + \tau_n (\mathcal{A}(\w_h^n) + \mathcal{R}(\w_h^n) + \mathcal{D}(\w_h^n)).$$
	This is the only explicit method, we apply. It requires small time steps for stability reasons due to the explicit discretization of diffusion. Indeed, $\tau_n = \mathcal{O}(\max_i h_i^2)$.

\item[\textnormal{\textit{CND:}}]
	A Crank-Nicolson type method with
	$$\w_h^{n+1} - \frac{\tau_n}{2} \mathcal{D}(\w_h^{n+1}) = \w_h^n + \tau_n \left(  \frac12 \mathcal{D}(\w_h^n)  + \mathcal{R}(\w_h^n) + \mathcal{A}(\w_h^n) \right).$$
	Since we approximate the diffusion term implicitly, we can choose time steps according to the CFL condition \eqref{CFL} alone.
	
\item[\textnormal{\textit{ROSENBROCK:}}]
	A general $s$-stage linearly implicit method or Rosenbrock method takes the form:
	\begin{equation}\label{limplicit}\left\{
		\begin{aligned}
			\w_h^{n+1} &= \w^n_h + \tau_{n} \sum_{j=1}^s b_j k_j, \\
			(Id - a_{j,j}\tau_{n}~J)k_j &= g(\w^n_h + \tau_{n}~\sum_{\nu=1}^{j-1}(a_{j,\nu}+\gamma_{j,\nu})k_\nu)-\tau_n J \sum_{\nu=1}^{j-1}\gamma_{j,\nu}k_\nu, ~~j=1,\dots,s.
	 \end{aligned}
	\right.
	\end{equation}
	for given lower triangular matrices $A=(a_{i,j})_{i,j=1}^s,~\Gamma=(\gamma_{i,j})_{i,j=1}^s$, a vector $b$ and
	$$g(\w_h)=\mathcal{A}(\w_h^n) + \mathcal{R}(\w_h^n) + \mathcal{D}(\w_h^n).$$
	If $A,~b$ satisfy specific algebraic conditions, then high order of consistency can be reached with these methods. Stability properties can be achieved by selecting reasonable values for the parameter matrix $\Gamma$ and a suitable approximation $J$ of $\frac{\partial}{\partial \w_h}(g(\w_h)$  \cite{gerisch2006robust}. We choose $J = \frac{d}{d\w_h}( \mathcal{R}(\w_h^n) + \mathcal{D}(\w_h^n))$ since we assume the taxis discretization to be stable in explicit methods.
	
\item[\textnormal{\textit{ROS2:}}]
	An L-stable, second order consistent two stage Rosenbrock method ($s=2$), which has already been used for applications in reaction-diffusion-taxis systems in \cite{gerisch2002operator}. It is given by the coefficients
	\begin{equation} \label{ROS2_coeffs}
		A = \begin{pmatrix}1- \frac{\sqrt2}{2} & 0  \\ \sqrt2-1 & 1-\frac{\sqrt2}{2} \end{pmatrix},~
			\Gamma = \begin{pmatrix}0 & 0  \\ 2-\sqrt2 & 0 \end{pmatrix},~
		b = \begin{pmatrix}
		\frac{1}{2} & \frac{1}{2}
     \end{pmatrix}^T .
	\end{equation}
	
\item[\textnormal{\textit{ROS3:}}]	
	An L-stable \cite{hundsdorfer2003numerical} third order Rosenbrock method ($s=3$) with coefficients
	\begin{align*}
	  A = \begin{pmatrix}a & 0 & 0 \\ 0.5 & a & 0 \\
						0.5 & 0.5 & a \end{pmatrix},
		&~~\gamma_{2,1} = - (3a + \gamma_{3,1}+\gamma_{3_2}),\\
	a = 1-\frac{1}{2}\sqrt2 \cos(\theta)+\frac12 \sqrt6 \sin(\theta),
		&~~\gamma_{3,2} = \frac12 - 3a,\\
	b = \begin{pmatrix}
		\frac{1}{3} & \frac{1}{3} & \frac{1}{3}
   		 \end{pmatrix}^T,		
		&~~ \gamma_{3,1} = -\frac{1}{1+2\gamma_{3,2}}\left( 6a^3-12a^2 + 6(1+\gamma_{3,2})a + 2\gamma_{3,2}^2 - \frac{1}{2}\right),\\
	 \theta = \frac{1}{3}\arctan\left(\frac{\sqrt2}{4}\right).
	\end{align*}

\item[\textnormal{\textit{ROS3-ATC:}}]
	A ROS3 method with adaptive time step control.

\item[\textnormal{\textit{STRANG:}}]
	A second order splitting approach. Let $\Xi_\mathcal{F}(\tau) y$ be a numerically approximated solution of the initial value problem
	\[ \partial_t \w_h = \mathcal{F}(\w_h), \quad \w_h(0) = y, \label{fracInitial} \]
	at $t = \tau$, for any operator introduced before,
	$\mathcal{F}\in\{ \mathcal{D}, ~\mathcal{R},~\mathcal{A}\}.$
	The Strang-splitting method yields
	\begin{equation}\label{TDRDT}
		\w_h^{n+1} = \Xi_\mathcal{T}\left(\frac{\tau_n}{2}\right)\,\Xi_\mathcal{D}\left(\frac{\tau_n}{2}\right)\,\Xi_\mathcal{R}(\tau_n)\,
		\Xi_\mathcal{D}\left(\frac{\tau_n}{2}\right)\,\Xi_\mathcal{T}\left(\frac{\tau_n}{2}\right)\,\w_h^n.
	 \end{equation}
	This method is known to be second order in time, if the methods used to compute $\Xi_\mathcal{F}(\tau_n)$, are at least second order. We choose the fourth order Runge-Kutta method \cite{hairer1991solving} for the reaction- and taxis-step and the TR-BDF2 method (modified trapezoidal rule with the second order backward differential formula) for the diffusion-terms. The TR-BDF2 can be written as an imlicit Runge-Kutta method, coefficients can be found in Table \ref{TR-BDF2}.
	\begin{table}
	\centering
	\renewcommand{\arraystretch}{1.2}
	  \begin{tabular}{>{$} c <{$}|>{$} c <{$} >{$} c <{$} >{$} c <{$}}
	0 & 0 & & \\ \frac{1}{2} & \frac{1}{4}& \frac{1}{4} &\\
	1 & \frac{1}{3} & \frac{1}{3}& \frac{1}{3} \\
	\hline 1 & \frac{1}{3} & \frac{1}{3}& \frac{1}{3}
	\end{tabular}
	\caption{Butcher tableau for the simplified TR-BDF2 method.}
	\label{TR-BDF2}
	\end{table}

\item[\textnormal{\textit{STRANG-CND and STRANG-IR:}}]
	Two variants of the \textit{STRANG} method: \textit{STRANG-CND} uses the Crank-Nicolson method for the diffusion term instead of TR-BDF2 and \textit{STRANG-IR} employs the linearly implicit second order method \textit{ROS2} for the reaction term instead of the explicit Runge-Kutta method.

\item[\textnormal{\textit{IMEX2 and IMEX3:}}]
	Second and third order implicit-explicit methods following \cite{pareschi2005implicit}. We consider a splitting of the ordinary differential equation into an implicit part $\mathcal{I}$ and an explicit part $\mathcal{E}$,
	$$g(\w_h)=\mathcal{A}(\w_h^n) + \mathcal{R}(\w_h^n) + \mathcal{D}(\w_h^n) = \mathcal{I}(\w_h) + \mathcal{E}(\w_h),$$
	where $\mathcal I= \mathcal D$ and $\mathcal E =\mathcal A+R$, and apply an explicit Runge-Kutta method for the explicit part and a coupled diagonally implicit Runge-Kutta-method to the implicit part resulting in an implicit-explicit (IMEX) scheme. A general $s$-stage scheme reads
	\begin{equation}\label{IMEX_RK}
		\begin{dcases}
		\textbf E_j 	= \mathcal{E}(\textbf W_j), 																			& i=1,\dots,s,\\
		\textbf I_j		= \mathcal{I}(\textbf W_j	),																			& i=1,\dots,s,\\
		\textbf W_i^* 	= \w_h^n + \tau_{n} \sum_{j=1}^{i-2}\bar a_{i,j}\textbf E_j + \tau_{n}~\bar a_{i,i-1}\textbf E_{i-1}, 	& i=1,\dots,s,\\
		\textbf W_i 	= \textbf W_i^* + \tau_{n}\sum_{j=1}^{i-1}a_{i,j} \textbf I_j + \tau_{n}a_{i,i}\textbf I_i,				& i=1,\dots,s,\\
		\w_h^{n+1}=\w^n_h + \tau_{n}~\sum_{i=1}^s \bar b_i \textbf E_i + \tau_{n}~\sum_{i=1}^s b_i \textbf I_i,
		\end{dcases}
	\end{equation}
	where the explicit scheme is given by $\bar b,~\bar A$ and the diagonally implicit scheme by $b,~A$.

	We concentrate on two particular schemes. The first one is the IMEX-Midpoint scheme. Though it only uses one implicit stage for the diffusion term, it is second order accurate in time \cite{pareschi2005implicit}. This makes it less expensive than the Strang-splitting. Further, we consider a third order scheme constructed in such a way that it  fulfills several stability conditions, e.g. A- and L- stability \cite{christopher2001additive}. Coefficients for both methods can be found in Tables \ref{IMEX2} and \ref{IMEX3}. 

	\begin{table}
	\hspace{3cm}
	\renewcommand{\arraystretch}{1.2}
	  \begin{tabular}{>{$} c <{$}|>{$} c <{$} >{$} c <{$}}
	0 & 0 & 0 \\ \frac{1}{2} & \frac{1}{2} & 0 \\ \hline & 0&1
	\end{tabular}\hspace{3cm}
	 \begin{tabular}{>{$} c <{$}|>{$} c <{$} >{$} c <{$}}
	0 & 0 & 0 \\ \frac{1}{2} & 0&\frac{1}{2} \\ \hline & 0&1
	\end{tabular}
	\caption{Tableau for the IMEX-Midpoint scheme (\textit{IMEX2}). Coefficients $\bar A,~\bar b, \bar c$ for the explicit scheme on the left and $A,~b,~c$ for the diagonally implicit scheme on the right, respectively.}
	\label{IMEX2}
	\end{table}
	
	\begin{table}
	\centering
	\renewcommand{\arraystretch}{1.2}
	 \begin{tabular}{c|c c c c}
	$0$ & & & &\\
	$\frac{1767732205903}{2027836641118}$ & $\frac{1767732205903}{2027836641118} $& & & \\
	$\frac{3}{5}$ & $\frac{5535828885825}{10492691773637}$ & $\frac{788022342437}{10882634858940} $& & \\
	$1$ & $\frac{6485989280629}{16251701735622}$ & $-\frac{4246266847089}{9704473918619} $& $\frac{10755448449292}{10357097424841}$ & \\ \hline
	 & $\frac{1471266399579}{7840856788654}$ & $-\frac{4482444167858}{7529755066697}$ & $\frac{11266239266428}{11593286722821}$ & $\frac{1767732205903}{4055673282236}$
	\end{tabular}
	\vspace{0.4cm}
	\renewcommand{\arraystretch}{1.2}
	 \begin{tabular}{c|c c c c}
	  0 & 0 & & & \\
	$\frac{1767732205903}{2027836641118}$ &$\frac{1767732205903}{4055673282236}$ & $\frac{1767732205903}{4055673282236}$ & & \\
	$\frac{3}{5}$ &$\frac{2746238789719}{10658868560708}$ & $-\frac{640167445237}{6845629431997}$ & $\frac{1767732205903}{4055673282236}$ & \\
	1 & $\frac{1471266399579}{7840856788654}$ & $-\frac{4482444167858}{7529755066697}$ & $\frac{11266239266428}{11593286722821} $& $\frac{1767732205903}{4055673282236}$\\
	\hline
	 & $\frac{1471266399579}{7840856788654}$ & $-\frac{4482444167858}{7529755066697}$ & $\frac{11266239266428}{11593286722821} $& $\frac{1767732205903}{4055673282236}$
	 \end{tabular}
	\caption{Tableaux for the explicit (first tableau) and the implicit part (second tableau) of the third order IMEX-Runge-Kutta scheme (\textit{IMEX3}).}
	\label{IMEX3}
	\end{table}
	
\item[\textnormal{\textit{IMEX3-ATC:}}]
	Third order IMEX method with adaptive time step control. This method is applied in two variants: \textit{IMEX3-ATC-UPWIND1} uses first order upwind fluxes ($s_{i}^{(j)}(c_h)=0$ in \eqref{numFlux}). \textit{IMEX3-ATC-IR} treats reaction terms implicitly and uses $\mathcal{I}=\mathcal{D}+\mathcal{R}$.
\end{description}

	\paragraph{Adaptivity in time.}
	 Adaptive time step control is done conventionally by employing an additional lower order scheme in order to calculate a local error estimate
	\begin{equation}
		\epsilon_n = \| \w_h^{n+1} -\w_h^{n+1,\text{ low}}\|_\infty.
	\end{equation}
	
	The approximation $w_h^{n+1}$ is accepted if
	\begin{equation}
	\epsilon_n< \epsilon_n^\text{tol} = \max\{10^{-6},~10^{-6}\,\|w_h^n\|_1\},
	\end{equation}
	otherwise the same time step is repeated for a smaller value of $\tau_n$ which we get by multiplying the old value by $0.9~\sqrt[\leftroot{0}\uproot{1}3]{\frac{e^\text{tol}_n}{e_n}}$.
	
	We consider adaptive time step control for the third order linearly implicit scheme and the third order IMEX-Runge-Kutta scheme. Both of them are third order methods and feature an embedded method of second order. Hence, lower order approximations $(\w_h^{n+1,\text{ low}})$ can be obtained without much additional computational costs. Weights $\beta$ which replace the regular weights $b$ in the lower order embedded schemes are given by
	$$\bar \beta = \beta =\begin{pmatrix} \frac{2756255671327}{12835298489170} & -\frac{10771552573575}{22201958757719} &
		\frac{9247589265047}{10645013368117} &  \frac{2193209047091}{5459859503100}\end{pmatrix}^T,$$
	for the IMEX3 scheme, and by
	$$ \beta = \begin{pmatrix} \frac12 &\frac12 & 0\end{pmatrix}^T,$$
	for the ROS3 scheme.

	\paragraph{Adaptivity in space.}
	We describe the adaptive mesh refinement we use in the case $d=1$. Mesh cells now depend on time as well
	$$\Omega =\bigcup_{i=1}^{N_n} C_i^n, \quad |C_i^n|=h_i^n.$$
	We consider the \textit{h-refinement} method for the refinement of the mesh and proceed as follows, see also \cite{kroner2000posteriori, ohlberger19999adaptive, puppo2012adaptive}. We prescribe a \textit{monitor function}, that depends on the numerical solution, and we evaluate it on every cell $i$, i.e. $M_i(\w_h^n)$. The monitor function will be compared against two properly chosen coarsening/refinement thresholds $C_\text{coa}<C_\text{ref}$. We also assume a uniform initial grid and set $L_i^n$ to be the level of refinement of the cell $i$ (zero for the initial grid). The refinement/ coarsening strategies read as follows
\begin{description}
	\item[\textnormal{\textit{Refinement.}}]
		If $M_i(\w_h^n)>C_\text{ref}$ the cell is marked for refinement. If the cell is of refinement level $k$ it is bisected into two isodynamous daughter cells of $k+1$ level of refinement. Then approximate values of the monitor function, using the mother cell and the neighbouring cells, are computed on every daughter cell. The refinement process is repeated $n_\text{ref}$ times.
	\item[\textnormal{\textit{Coarsening.}}]
		If $M_i(\w_h^n)<C_\text{coa}$ the cell is marked for coarsening, and if both daughter cells of the same mother cell are marked for coarsening, they merge into a single cell. The approximate value of the monitor function on the mother cell is computed by of the daughter cells' monitor function. The coarsening process is repeated $n_\text{coar}$ times.
\end{description}
We also prescribe a maximal refinement level $L_\text{max}$.

The monitor functions that we consider in this work are the following
\begin{enumerate}
	\item The discrete gradient of $c$:
	\begin{equation}\label{monitor:gradient}
 		M_i(\w_h) =\max\left\{ \left|2\frac{c_{i+1}-c_i}{h_{i+1}+h_i}\right|, \left|2\frac{c_i - c_{i-1}}{h_i + h_{i-1}}\right| \right\}.
	\end{equation}
	\item Local approximations of the discretization errors of characteristic velocities:
	\begin{equation}\label{monitor:charSpeeds}
		M_i(\w_h) = \max \left\{\left|\mathcal{P}_{i-1/2}-\mathcal{P}_{i+1/2}^\text{low}\right|,\  \left|\mathcal{P}_{i+1/2}-\mathcal{P}_{i+1/2}^\text{low}\right|\right\},
	\end{equation}
	where the lower order approximations of the characteristic velocities $\mathcal{P}_{i+1/2}^\text{low}$ are calculated using two point approximations of the first derivative.
\end{enumerate}

\section{Numerical Experiments}\label{experiments}
We present results of numerical simulations\footnote{All the numerical experiments were conducted using MATLAB.}and compare the performance of the methods we introduced previously, and demonstrate the capability of \textit{h-refinement}.

In order to illustrate the dynamics of the system \eqref{chaplolsystem}, and to compare different time integration techniques we present simulation results, for benchmark problems similar to those in \cite{andasari2011mathematical}. Therefore, we first consider the system \eqref{chaplolsystem} on a one-dimensional interval $\Omega$ together with homogeneous Neumann boundary conditions.

\subsection{Experiment I}
The parameters are chosen according to the parameter set $\mathcal{P}$ \eqref{params}. The following initial conditions are assumed:
\begin{equation}\label{exp1dinitial}
\left\{
\begin{aligned}
c_0(x) =~&\exp\left(\frac{-x^2}{\varepsilon}\right),\\
v_0(x) = ~&1- \frac 1 2 \exp\left(\frac{-x^2}{\varepsilon}\right),\\
u_0(x) = ~&\frac 1 2 \exp\left(\frac{-x^2}{\varepsilon}\right),\\
p_0(x) = ~&\frac {1}{20} \exp\left(\frac{-x^2}{\varepsilon}\right),\\
m_0(x) = ~&0,
\end{aligned}\right.\ , \qquad\qquad x\in(0,10),
\end{equation}
where $\varepsilon= 5 \cdot 10^{-3}$. The initial conditions can be interpreted as an accumulation of cancer cells $c$, which start their invasion from the left boundary of the domain. The extracellular matrix $v$ is mostly intact, except for the location of the cancer cell accumulation. Activation of plasmin $m$ has not taken place up to $t=0$, but urokinase $u$ and a smaller amount of urokinase inhibitor type 1, $p$, is already on the spot everywhere, where cancer cells are located.

Since we have observed a reliable grid convergence with the IMEX3 time integration method, we chose this method to compute reference solutions.

Figure \ref{exp1fig} shows the computed time evolution on the domain $\Omega=(0,10)$. We can see a cluster of cancer cells, which travels to the right side of the domain, degenerating the ECM. At areas, where vitronectin is already degenerated to a small amount, new clusters, which take the form of peaks in the cancer cell densities, emerge. This can already be seen at $t=25$. The number of these peaks as well as their heights in $c$ vary in time. Clusters not only emerge, but they also move and merge. After approximately nine days of biological time ($t=75$) almost half of the domain is invaded by the cancer cells and even when the entire domain is invaded, cancer cell density exhibits a  dynamically heterogeneous spatio-temporal behavior, which can be seen by comparison at $t=300$ and $t=500$. The enzymes of the uPA-system, which regulate the process of invasion, do not develop clusters and take densities between $0$ and $1$ all over the period $t\in[0,500]$. The inhibitor PAI-1 density stays smooth mostly, while uPA and plasmin develop spiky solutions, as they are more directly influenced by the cancer cells. Due to its accuracy and relatively low computational costs, we chose the IMEX3 method as a favorite method, to compute the solutions.
\begin{figure}
\centering
\subfigure[$t=5$]{
\includegraphics[width=0.47\textwidth]{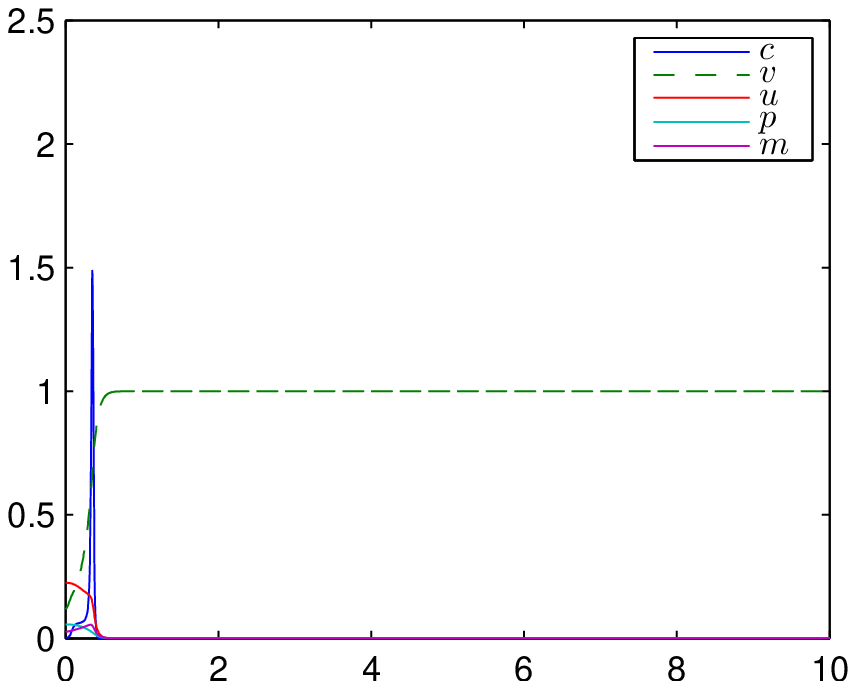}
}
\subfigure[$t=25$]{
\includegraphics[width=0.47\textwidth]{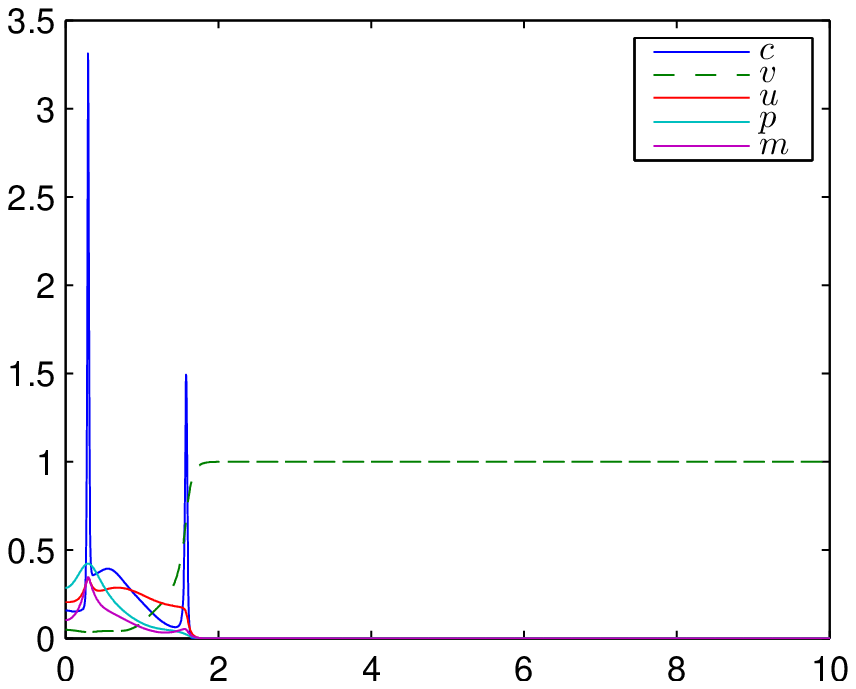}
}
\subfigure[$t=75$]{
\includegraphics[width=0.47\textwidth]{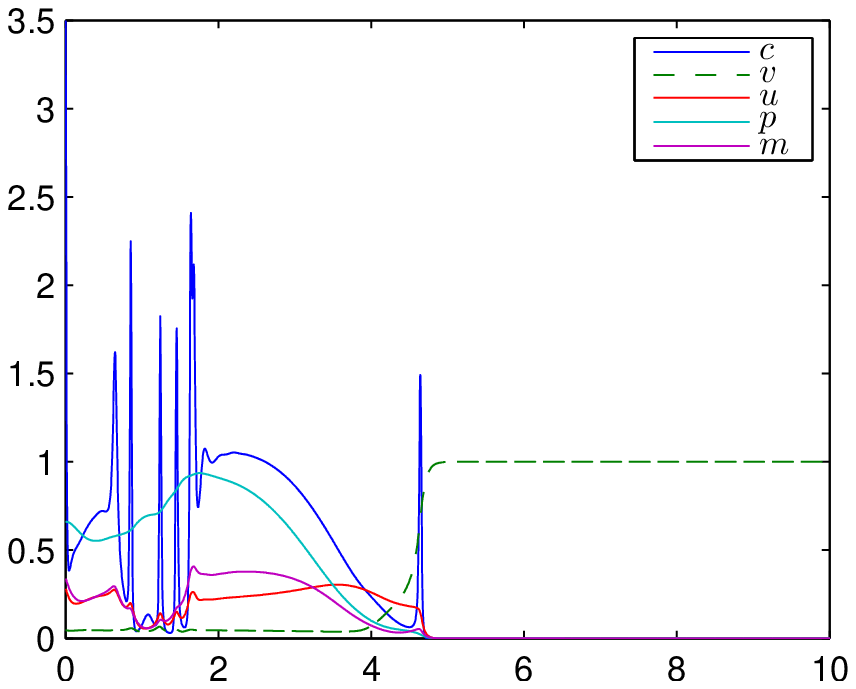}
}
\subfigure[$t=150$]{
\includegraphics[width=0.47\textwidth]{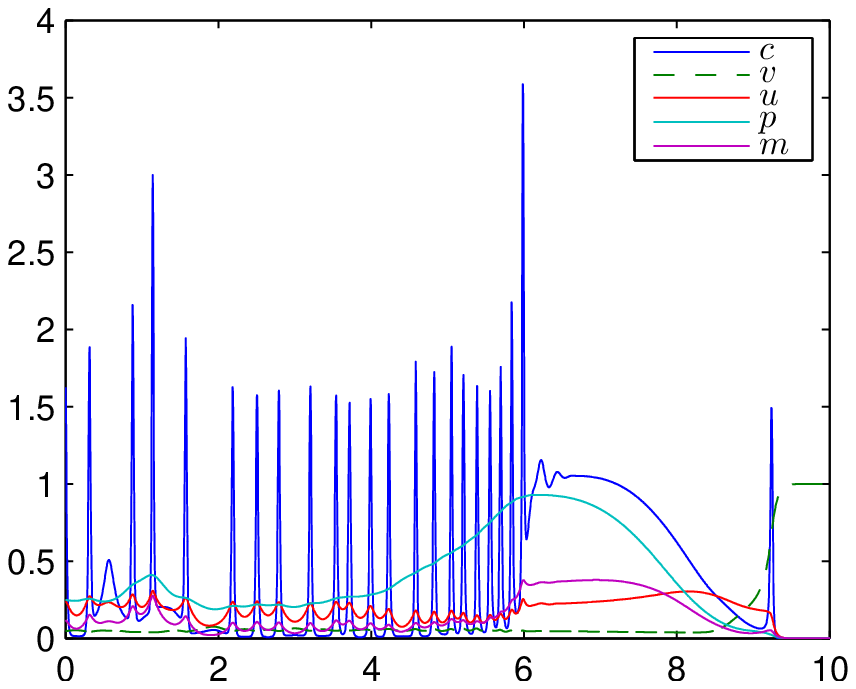}
}
\subfigure[$t=300$]{
\includegraphics[width=0.47\textwidth]{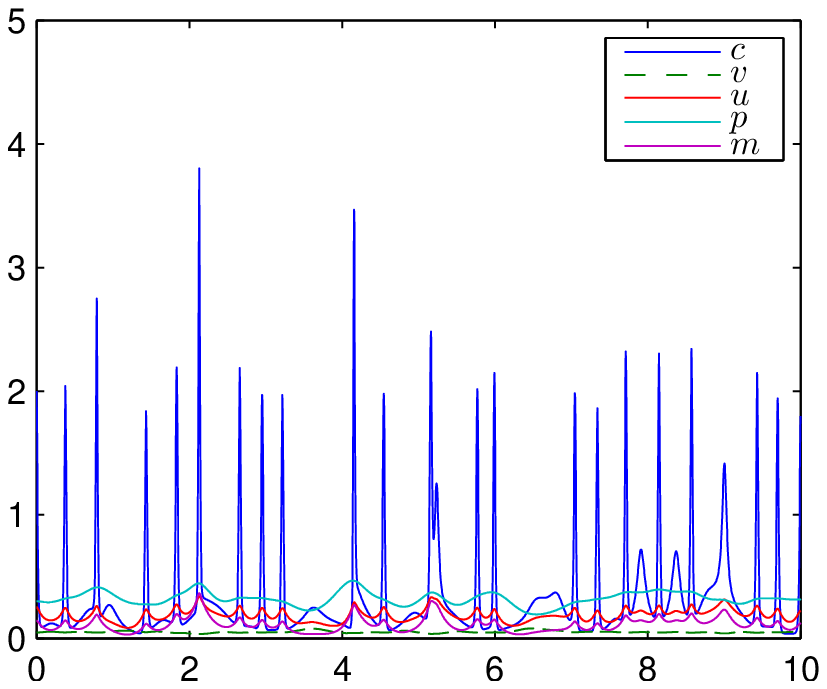}
}
\subfigure[$t=500$]{
\includegraphics[width=0.47\textwidth]{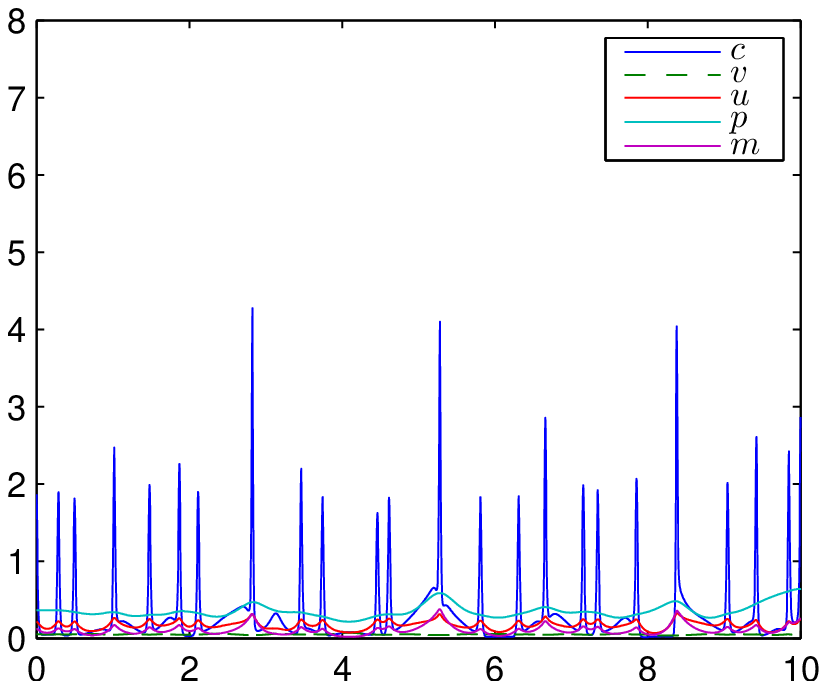}
}
\caption{Results of experiment I with parameterset $\mathcal{P}$ with $D_c=3.5~10^{-4}$.}
\label{exp1fig}
\end{figure}

In order to compare different time integration methods, we consider the narrowed domain $(0,5)$ and the final time $T=60$. We will study how the methods, described in the previous section, perform in this case. For comparison reasons, we consider a fixed Courant number of CFL=0.49 for every method. The influence of different Courant numbers is studied later, cf. Figure \ref{difference_CFL}.

Since the exact solution is not known, we compute a reference solution $\w^\text{ref}$ on a grid with $50\,000$ uniform cells $C_i^\text{ref},~i=1,\dots,50\,000$. We are only interested in a solution at time $T$ and thus we drop the time index and denote by $w^\text{ref}_i$ the reference solution at $T=60$ on cell $C_i^\text{ref}$, its interpolant by $w^\text{ref}$ and its first component, the cancer cell density, by $c^\text{ref}$. In order to compare the accuracy of the introduced methods in space, we compute the discrete $L^1$-errors of the cancer cell densities:
\begin{equation}\label{discreteL1}
 E(N):=|c^\text{ref}-c^N|_{L^1_{disc}(\Omega)}=\sum_{\stackrel{C_i^N=(x_i-h/2,x_i+h/2],}{x_i~\in~ C_j^\text{ref}}}\operatorname{vol}(C_i^{N})~|c_i^N-c_j^\text{ref}|,
\end{equation}
where $w^N$ with its first component $\w^{N,(1)}=c^N$ is a numerical solution on $N$ cells $C^N_1,\dots,C_N^N$. We compute numerical solutions on $N$ cells for each method for
$$N\in \{100,~ 200,~ 400,~ 800,~ 1000,~ 2000,~ 3000,~ 4000,~ 5000\}$$
and plot $N$ against $E(N)$ in log-log scale in order to visualize the convergence of the method experimentally. Due to the limitations of the comparison of numerical simulations with reference solutions computed over very fine grids, the number of grid cells of the numerical solutions should not be more than $10 \% $ of the number of cells  of the reference. This is why we restrict our numerical experiments to $N=5000$.

\begin{figure}
\includegraphics[width=1\textwidth]{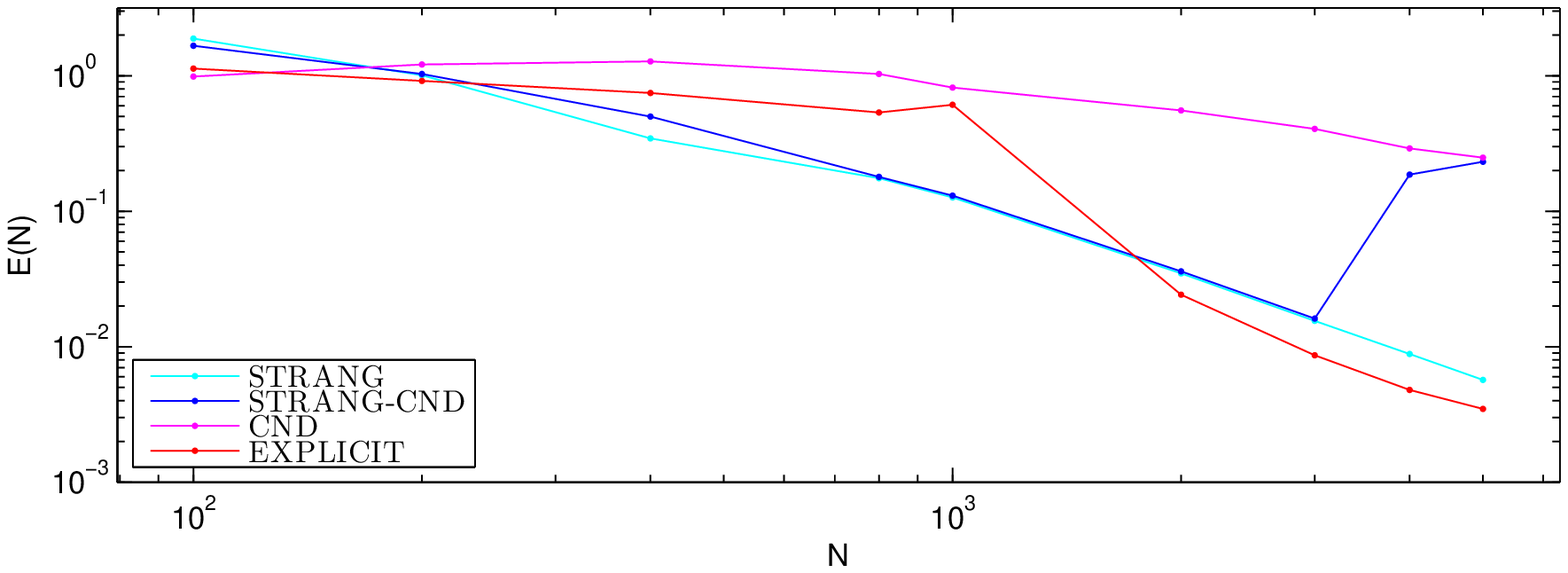}
\caption{Experimental convergence plot of splitting and first order methods in experiment I in log-log scale.}
\label{convergence_splitting}
\end{figure}

Figure \ref{convergence_splitting} indicates, that CND does not compute acceptable approximations. Though a first order convergence can be observed, the relative error is larger than $11\%$ even on grids with $N=5000$ cells. However, STRANG seems to converge second order in space and produces accurate solutions if about $1000$ cells are used. Due to the instabilities on fine meshes, STRANG-CND does not converge. The instabilities of the Strang-splitting, which employs  the Crank-Nicolson method for diffusion-terms, have been observed already in \cite{tyson2000fractional}. The good performance of the EXPLICIT method can be explained by the much smaller time steps it uses in order to be stable. The large amount of time steps, however, increases the computational costs.

Figure \ref{difference_splitting} demonstrates, that the application of implicit methods for reaction-terms in STRANG and IMEX3-ATC gives only negligible advantages over explicit reaction-terms on fine meshes, which further decrease with decreasing cell widths. In Figure \ref{convergence_ROS_IMEX} the second order convergence of IMEX3, ROS2 and ROS3 is demonstrated. All of them perform more accurately than STRANG, whereas IMEX3 gives the best results. The computational costs for IMEX3 are also less then the costs for ROS2 and ROS3, since the linearly implicit method handle reaction-terms implicitly, which makes them solve systems of linear equations with less sparse matrices. Surprisingly, the two stage Rosenbrock method gives a slightly better accuracy, than the three stage Rosenbrock method. The IMEX2 method develops instabilities on fine meshes and therefore does not converge. Though a slow convergence of the IMEX3 method with first order upwind fluxes can be observed, it does not produce accurate approximations.

\begin{figure}
\centering
 \includegraphics[width=1\textwidth]{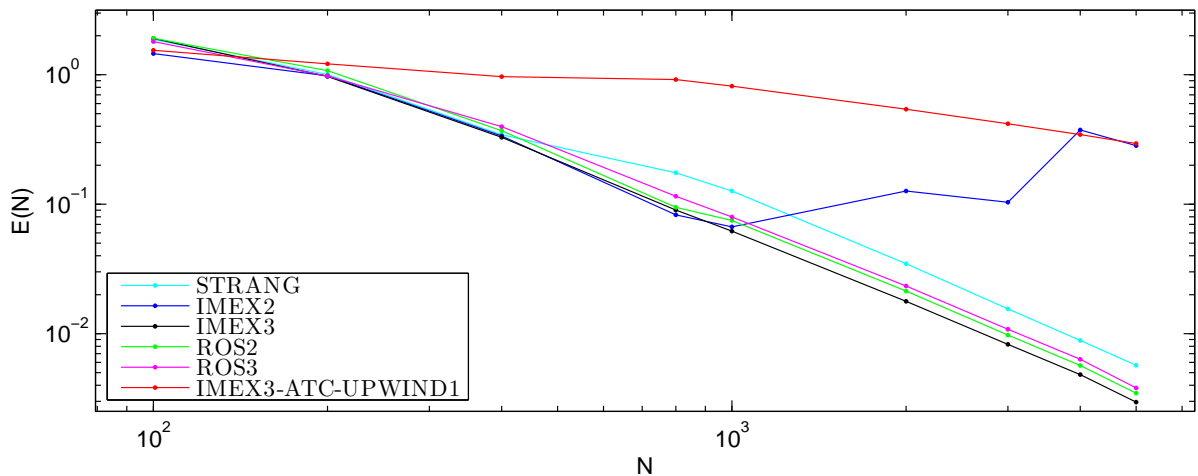}
\caption{Experimental convergence plot of different methods in experiment I in log-log scale}
\label{convergence_ROS_IMEX}
\end{figure}

\begin{figure}
\includegraphics[width=1\textwidth]{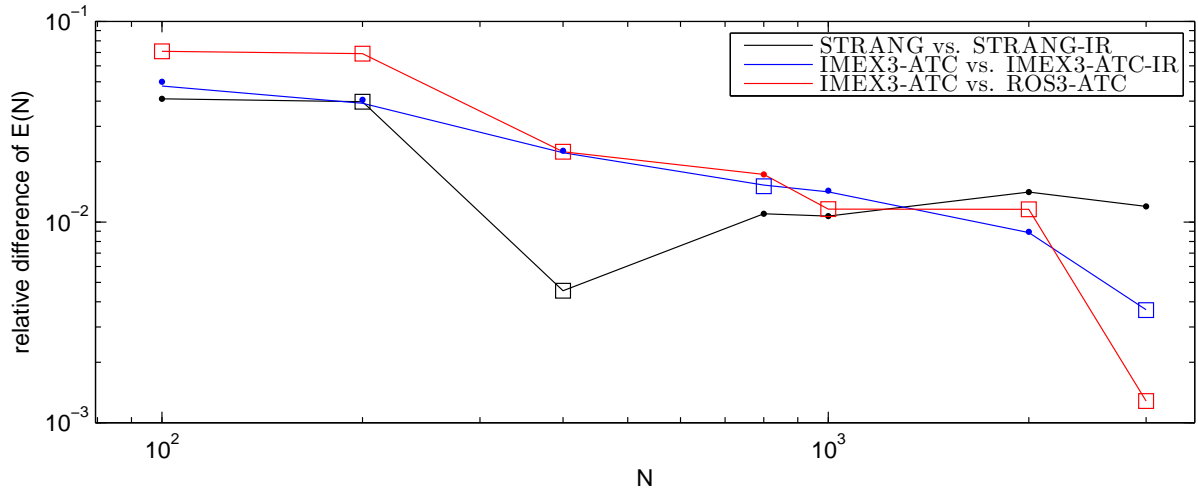}
\caption{Relative differences of the errors $|E^{M_1}(N)-E^{M_2}(N)|/E^{M_1}(N)$ of two methods $M_1$ and $M_2$ in log-log scale. Square markers symbolize, that the error of the first mentioned method is less than the error of the compared method.
Whether reaction terms are treated implicitly or explicitly has a minor, with grid size decreasing, impact on the error. }
\label{difference_splitting}
\end{figure}

Figure \ref{difference_splitting} shows that ROS3-ATC produces slightly more accurate solutions than IMEX3-ATC, but this advantage 	decreases with decreasing cell widths. Similarly we observe that the application of implicit methods for reaction-terms in STRANG and IMEX3-ATC gives only negligible advantages over explicit reaction-terms on fine meshes. Thus we state, that the advantage, which we gain, if we use implicit methods to handle the reaction-terms, is not worth the additional computational effort.
	
\begin{figure}
\includegraphics[width=1\textwidth]{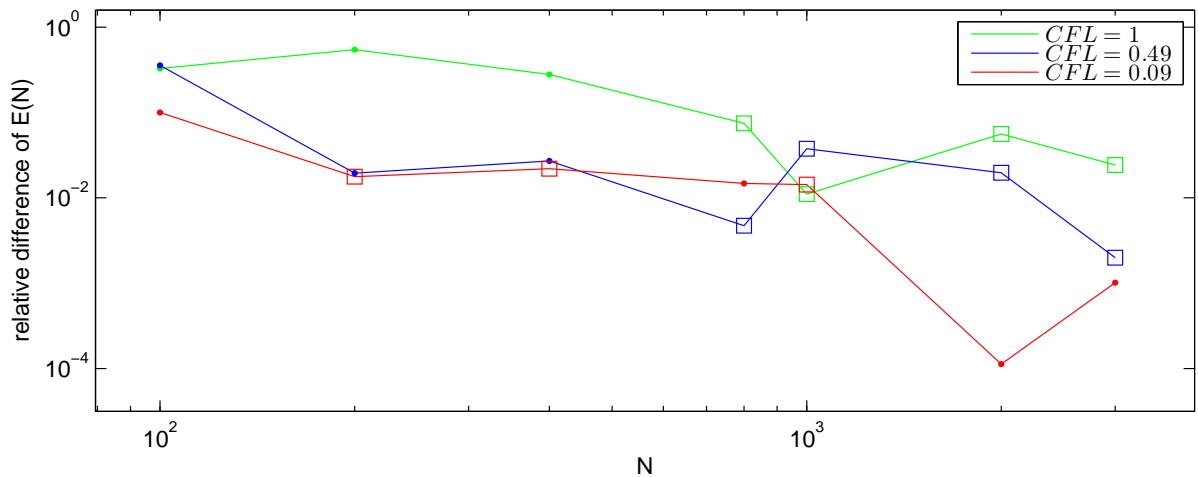}
\caption{The relative difference of IMEX3-ATC and IMEX3 with different Courant numbers in loglog scale. The decreasing relative errors indicate that Courant numbers smaller than $0.49$ do not improve the method worthwhile.} \label{difference_CFL}
\end{figure}

The influence of the CFL number on the accuracy of the IMEX3 method is much less than to the GODUNOV method and for sufficiently large number of cells, the IMEX3 method with CFL=0.49 produces almost as accurate results as its counterpart with adaptive time step control, see Figure \ref{difference_CFL}.

Table \ref{table_exp1} gives an overview of the experimental orders of convergence and the absolute discrete $L^1$-error of a sample solution with $2000$ cells as well as its computation time. The experimental order of convergence of a specified method is calculated based on its discrete $L^1$-errors for $N_1=2000$ and $N_2=5000$ cells
\begin{equation}\label{EOC}
	EOC = \frac{\log(E(N_1))-\log(E(N_2))}{\log(N_2)-\log(N_1)}.
\end{equation}
	
\begin{table}
\centering
		\begin{tabular}{|l|| >{$} c <{$} |>{$} c <{$}| >{$} c <{$}|}
	\hline
	& \textnormal{EOC} & \textnormal{Comperative CPU time} & \textnormal{discrete } L^1 \textnormal{ error}\\ \hline
	CND & 0.958 & 17.8 & 5.544\cdot 10^{-1} \\
	IMEX3-ATC-UPWIND1 & 0.691 & 48.7 & 5.414 \cdot 10^{-1}\\
 STRANG & 1.965 & 54.4 & 3.475 \cdot 10^{-2}\\
STRANG-IR & 1.965 & 177.5 &  3.426 \cdot 10^{-2}\\
EXPLICIT & 1.782 &  176.71 & 2.423 \cdot 10^{-2}\\
ROS3 & 2.044 & 246.39 & 2.337 \cdot 10^{-2}\\
 ROS2 & 2.025 & 165.1 & 2.131\cdot 10^{-2}\\
 IMEX3-ATC & 2.021 & 200.5 & 1.816\cdot 10^{-2}\\
 IMEX3-ATC-IR & 2.034 & 18931.6 & 1.800 \cdot 10^{-2} \\
 ROS3-ATC & 2.080 & 10773.5 & 1.795\cdot 10^{-2}\\
IMEX3 & 2.010 & 35.2 & 1.781 \cdot 10^{-2}\\
 \hline
\end{tabular}
\caption{Experimental orders of convergence, computation time and error of a sample approximation for $N=2000$ in experiment I}
\label{table_exp1}
\end{table}

\subsection{Experiment II}
In this experiment we investigate the performance of the methods in the case of a smooth solution.
\begin{figure}
\centering
\subfigure[$t=75$]{
\includegraphics[width=0.47\textwidth]{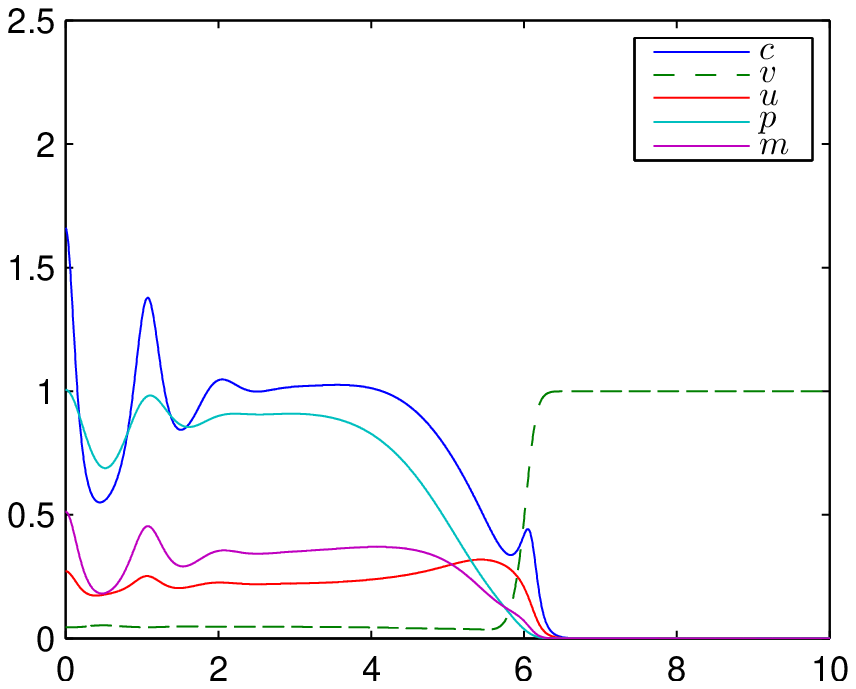}
}
\subfigure[$t=150$]{
\includegraphics[width=0.47\textwidth]{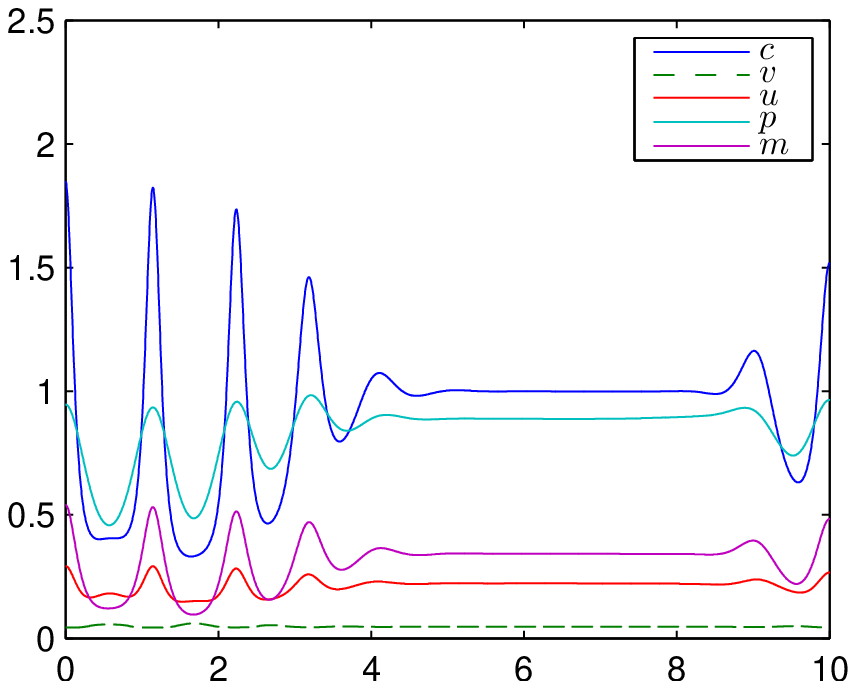}
}
\subfigure[$t=300$]{
\includegraphics[width=0.47\textwidth]{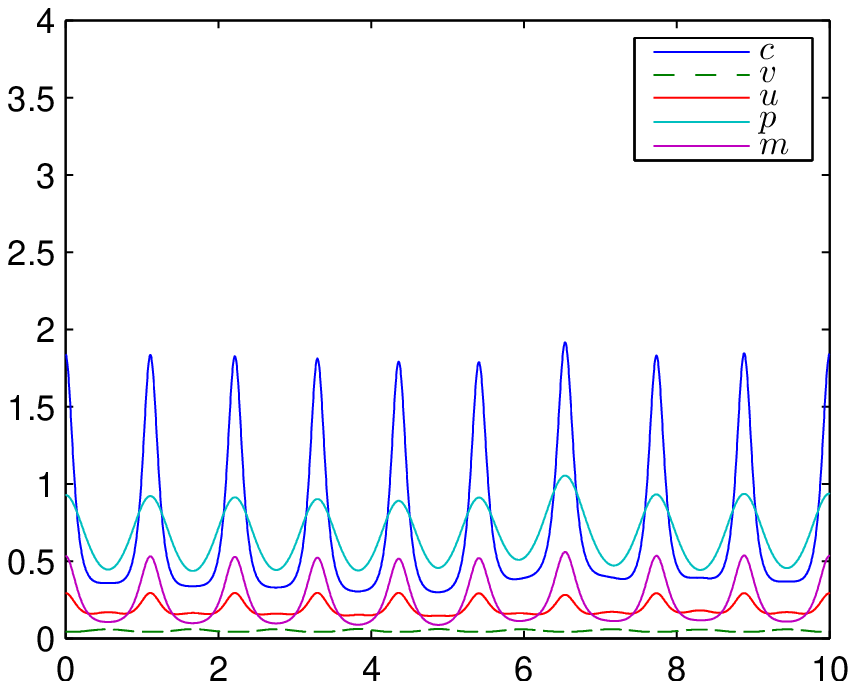}
}
\subfigure[$t=500$]{
\includegraphics[width=0.47\textwidth]{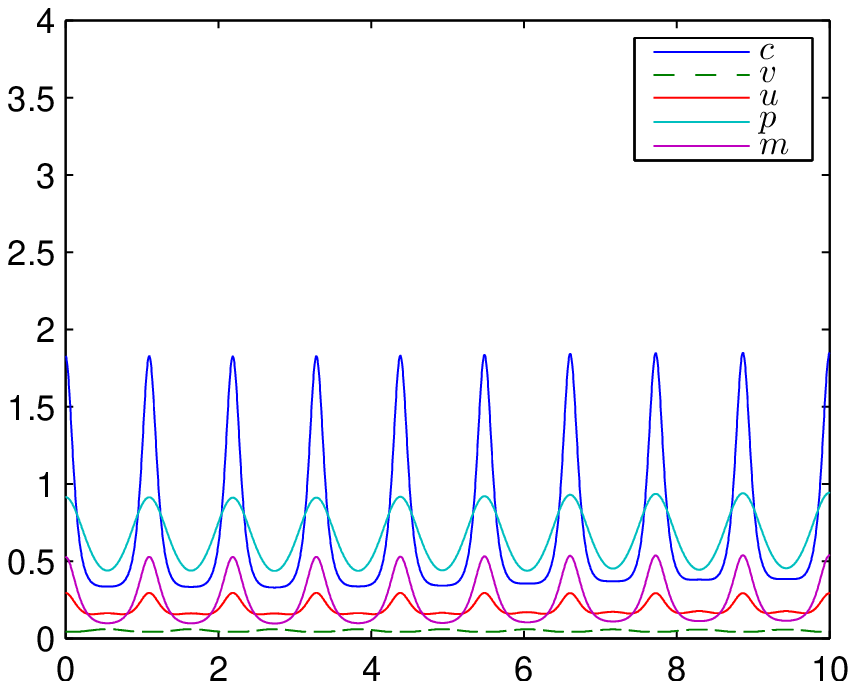}
}
\caption{Results of experiment II with parameterset $\mathcal{P}$, where the diffusion coefficient of the cancer cells is increased to $D_c=5.3~10^{-3}$}
\label{exp2fig}
\end{figure}
Figure \ref{exp2fig} exhibits the behavior of the solution of system \eqref{chaplolsystem} with homogeneous Neumann boundary conditions on the domain $\Omega=(0,10)$, initial conditions \eqref{exp1dinitial} and
parameter set $\mathcal{P}$ with an increased diffusion coefficient of the cancer cells,
\[D_c = 5.3~10^{-3}.\]
In order to estimate errors in this setting we make use of a reference solution $\w^\text{ref}$, computed with the IMEX3 method on $N=100\,000$ cells. Like before, we use a fixed Courant number $CFL=0.49$ and only compute solutions at the fixed time $T=50$. Test approximations are computed on the domain $\Omega=(0,5)$ distributed into $N$ cells with
\[N\in \{100,~ 200,~ 400,~ 800,~ 1000,~ 2000,~ 3000,~ 4000\},\]
and errors are calculated according to \eqref{discreteL1}.
\begin{figure}

\includegraphics[width=1\textwidth]{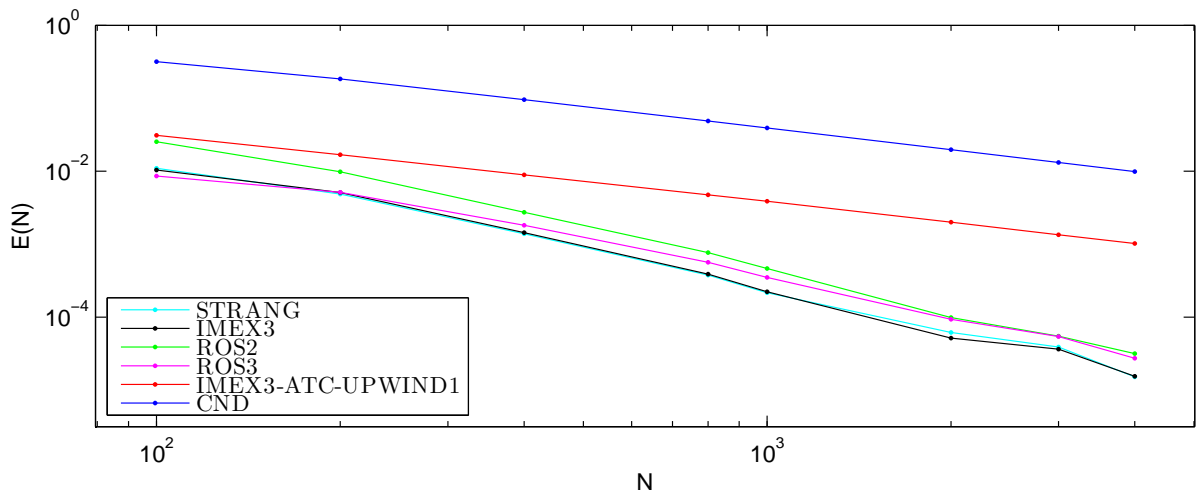}
\caption{Experimental convergence plot of the considered methods in experiment II in log-log scale} \label{difference_exp2}
\end{figure}

Figure \ref{difference_exp2} indicates that STRANG seems to be second order convergent in space. On the other hand, STRANG-CND suffers from instabilities even on coarse grids and does not converge as in the first experiment. The forward Euler method is unstable in this case although the time steps were limited according to the explicit discretization of diffusion. Similarly, instabilities occur in the IMEX2 method case where no further time step limitation is needed.

ROS2, ROS3 and IMEX3 appear to be second order accurate, while IMEX3 approximates best again. In this smooth case however, the Strang-splitting method performs similarly to IMEX3. The three stage Rosenbrock method is advantageous over the two stage Rosenbrock method, which reduces with the cell width. First order fluxes are capable to resolve the solution, as the slow convergence of IMEX3-ATC-UPWIND1 in Figure \ref{difference_exp2} demonstrates.

	\begin{figure}
	\centering
	\includegraphics[width=1\textwidth]{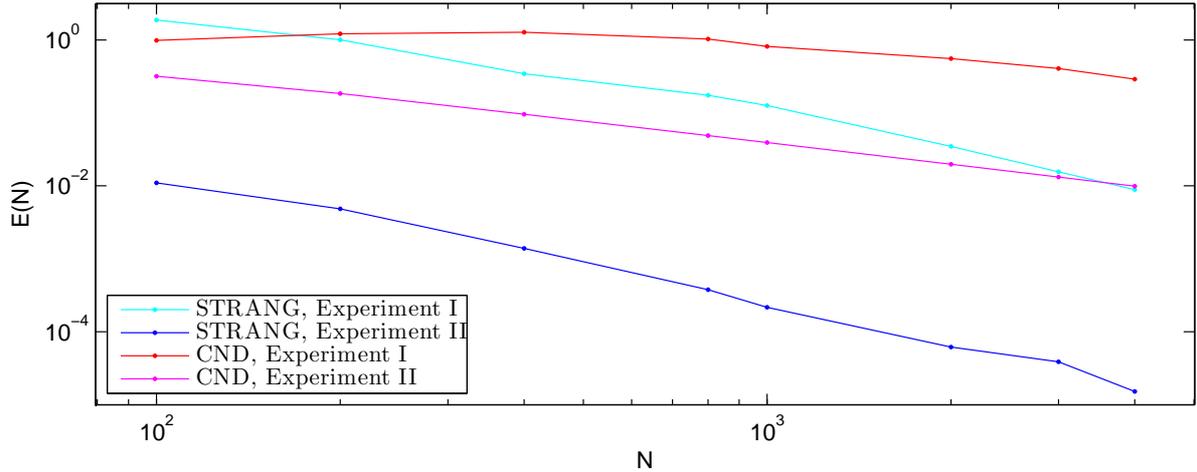}
	\caption{Experimental convergence plot of splitting methods in experiment I and experiment II in log-log scale}
	\label{convergence_12}
	\end{figure}

We notice in Experiment II, that the EOCs of the various methods, are the same as in Experiment I; there is though a significant drop in the actual  errors. This can be seen in  Figure \ref{convergence_12} and in Table \ref{table_exp2} by comparing the error convergence plots of the splitting methods. We also notice the marginal positive influence of adaptive time step control and implicit treatment of reaction-terms to the error. 

Table \ref{table_exp2} presents the EOC --computed according to \eqref{EOC}-- using $N_1=2000$ and $N_2=4000$. Again the discrete $L^1$-error refers to a sample approximation computed on $N=2000$ cells.

\begin{table}
\centering
\begin{tabular}{|l|| >{$} c <{$} | >{$} c <{$}|}
\hline
& \textnormal{EOC}  & \textnormal{discrete } L^1 \textnormal{ error}\\ \hline
CND & 0.935  & 1.980 \cdot 10^{-2} \\
 IMEX3-ATC-UPWIND1 & 0.651 & 2.007 \cdot 10^{-3} \\
 ROS2 & 1.906 & 9.853 \cdot 10^{-5} \\
 ROS3 & 1.884 & 9.338 \cdot 10^{-5} \\
STRANG-IR & 1.970 & 6.691\cdot 10^{-5} \\
 STRANG & 1.974 & 6.178\cdot 10^{-5} \\
IMEX3 & 1.882 & 5.167\cdot 10^{-5} \\
 IMEX3-ATC & 1.935 & 5.066\cdot 10^{-5} \\
 \hline
\end{tabular}
\caption{Experimental orders of convergence and error of a sample approximation for $N=2000$ in experiment II}
\label{table_exp2}
\end{table}

\begin{remark}
	In both, the larger and smaller diffusion case, the IMEX3 has produced approximations that are among the most accurate that we have achieved. At the same time, the computational cost of IMEX3 is short in comparison to other methods of the same accuracy. Therefore we consider IMEX3 in the non-uniform adaptive mesh case.
\end{remark}

\subsection{Adaptive mesh refinement}
Next, we investigate the benefits of adaptive mesh refinement by conducting experiment I again, starting on a grid with $400$ uniformly distibuted cells on $\Omega=(0,5)$. For our experiments we choose
IMEX3 as time integration method and define $n_\text{ref}=1, ~n_\text{coa}=3$ constantly. Further, we fix $L_\text{max} = 5$.
We compute the discrete $L^1$-errors over uniform and non-uniform grids using the  formula
\begin{equation}\label{discreteL1_time}
 E(t^n):=|c^\text{ref,n}-c^n|_{L^1_{disc}(\Omega)(t^n)}=\sum_{\stackrel{C_i^{N_n}=(x_i^n-h_i^n/2,x_i^n+h_i^n/2]}{x_i^n~\in~ C_j^{ref}}}\operatorname{vol}(C_i^{N_n})~|c_i^n-c_j^{ref,n}|,
\end{equation}
where $c^\text{ref}$ is the reference solution for the cancer cell density, computed by IMEX3 an a uniform mesh with $50\,000$ cells, and $c$ is the the solution obtained by the method whose error we want to compute.

We consider the absolute gradient of $c$ \eqref{monitor:gradient} with thresholds $C_\text{ref} = 55,~C_\text{coa} = 35$, and the estimated discretization error of the characteristic velocities \eqref{monitor:charSpeeds} with thresholds $C_\text{ref} = 7\cdot 10^{-4},~C_\text{coa} = 4\cdot 10^{-4}$ as monitor functions for the adaptation of the mesh. We refer to Figure \ref{refinement_example} for an impression of the numerical solution over and adaptively redefined grid, using the gradient as monitor function.

\begin{figure}
\subfigure[Experiment I at $t=40$]{
\includegraphics[width=1\textwidth]{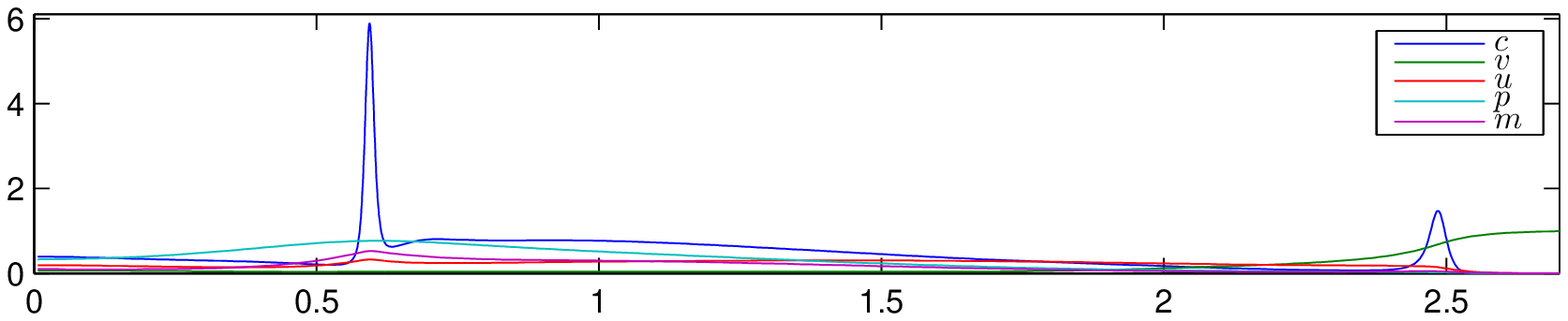}
}
\subfigure[Sizes of the cells]{
\includegraphics[width=1\textwidth]{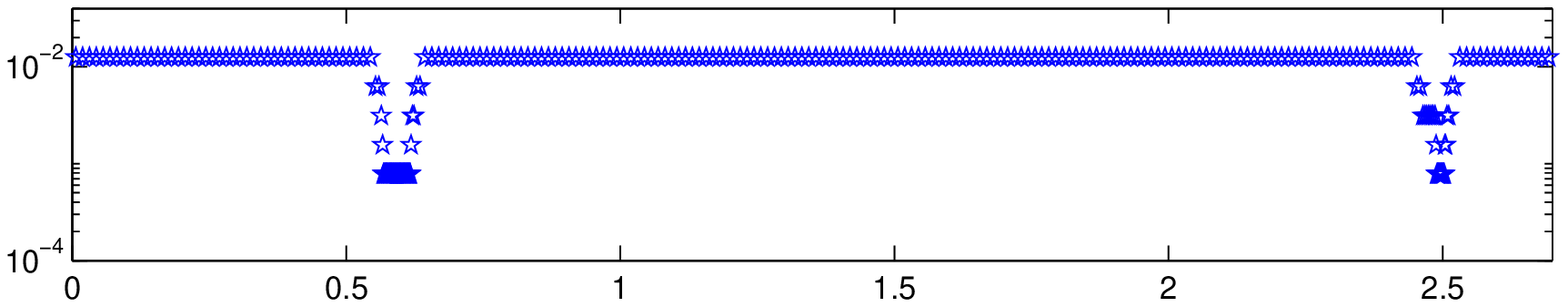}
}
\caption{A solution employing refinement of the cells according to the gradient of $c$. }
\label{refinement_example}
\end{figure}

\begin{figure}
\centering
\subfigure[Travelling of the front concentration and emerging of a second concentration for $0\leq t\leq 23$.]{
\includegraphics[width=0.48\textwidth]{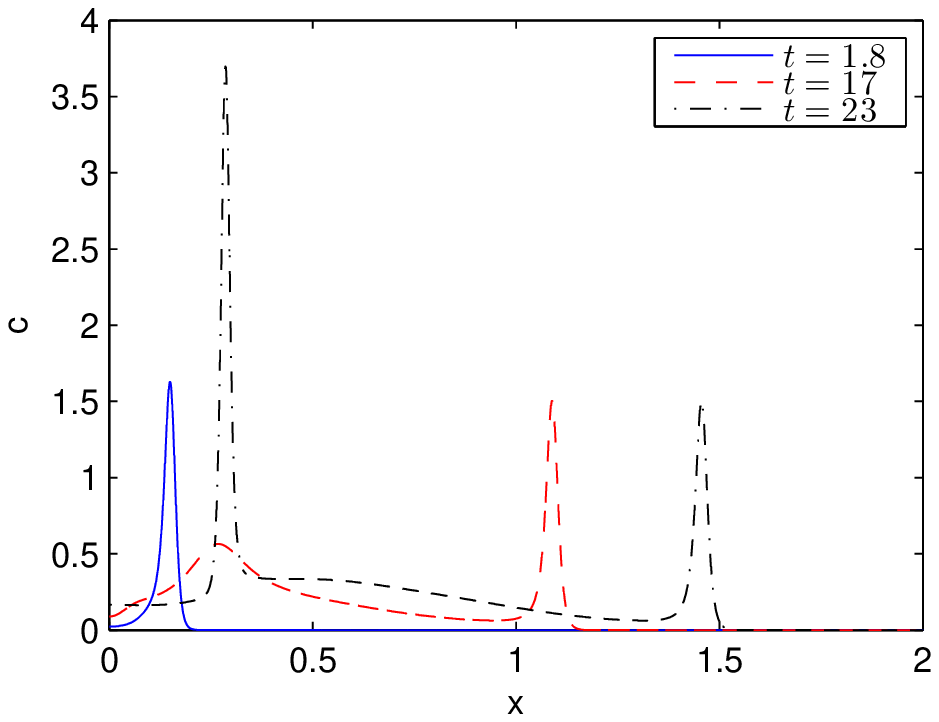}}
\subfigure[Emerging of a third concentration which merges with an existing one while $30 \leq t \leq40$.]{
\includegraphics[width=0.48\textwidth]{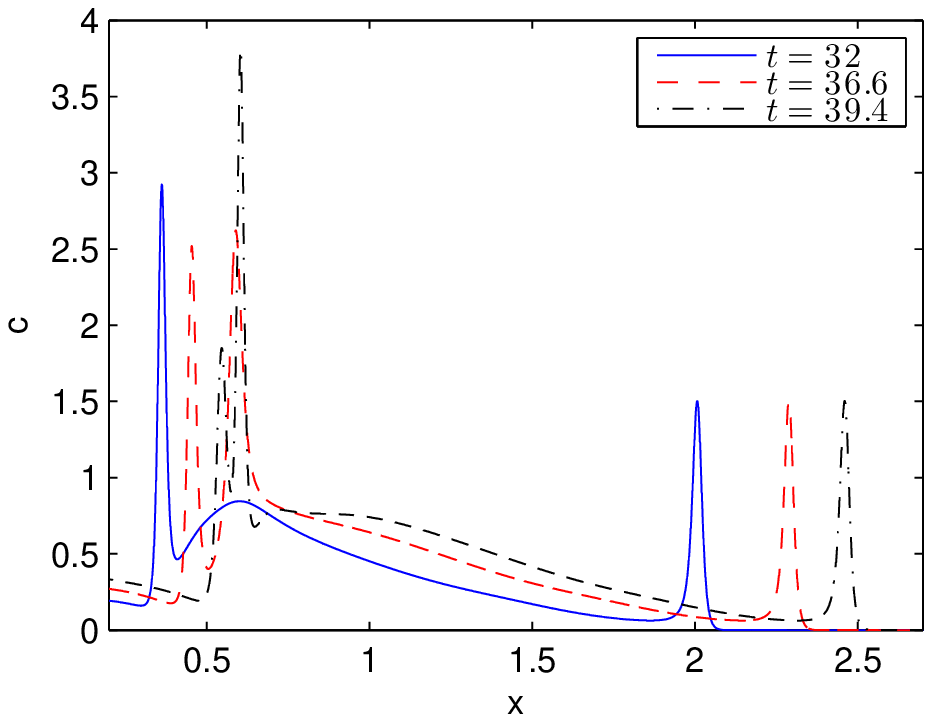}
}
\subfigure[Further merging and emerging of concentrations while $50\leq t \leq 60$]{
\includegraphics[width=0.96\textwidth]{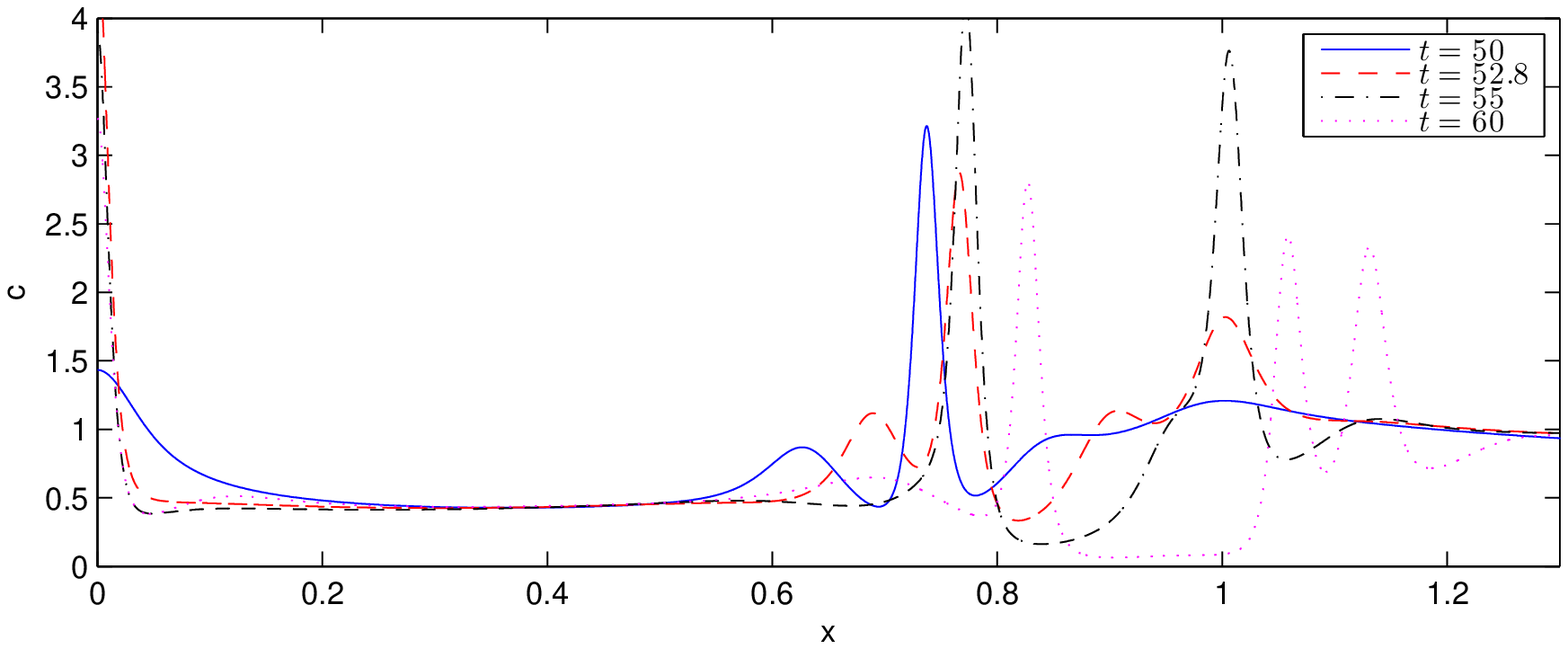}
}
\caption{Dynamics of the cancer cell concentration $c$ for $0\leq t \leq 60$ in experiment I.}
\label{experiment1_action}
\end{figure}

\begin{figure}
\centering
 \includegraphics[width=1\textwidth]{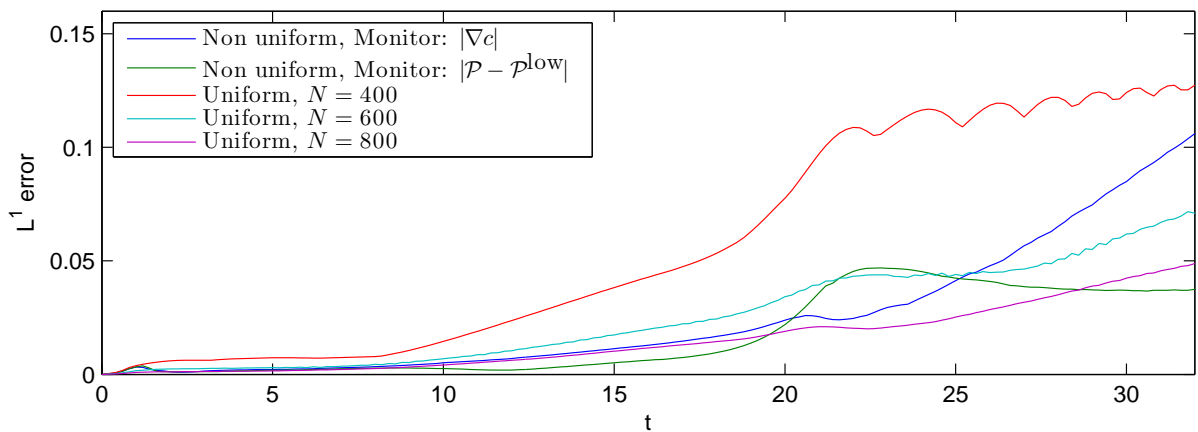}
\caption{A comparison of discrete $L^1$-errors as functions in time of uniform and nonuniform solutions for experiment I.}
\label{refinement_convergence_small_t}
\end{figure}

Figure \ref{experiment1_action} displays the dynamics that we aim to capture with the adaptive mesh refinement method that we employ. These involve moving, emerging, and merging of concentrations of the cancer cell densities.

Figure \ref{refinement_convergence_small_t} shows a visualization of the error \eqref{discreteL1_time} that the adaptive method, with the gradient as monitor function. The error is smaller than the uniform solution on $600$ cells up to the emerging of the second concentration (around $t=23$), albeit the refinement method uses less than $500$ cells. Employing discretization errors in monitor functions seems to be even more beneficial since the error of the second refinement method is during almost all times $t\in[0,35]$ less than the error of the uniform solution on $800$ cells. Note that the second refinement method does not use more than $500$ cells. However, around $t=40$, when concentrations of cancer cells merge, this advantage vanishes and the gradient based refinement methods produces smaller errors than the discretization based method. After $t=50$, when the dynamics become more complex, cf. Figure \ref{experiment1_action}(c), both refinement methods perform even worse than the uniform method on $400$ cells.

Further, we propose a modification of both refinement methods that aims for a better regularized structure of the grid. By ``smoothly refined grid'' we refer to a grid where the condition
\begin{equation}\label{smooth_grid}
 |L_i^n-L_{i+1}^n| \leq1,\quad i=1,\dots N_n-1,
\end{equation}
holds. To produce a smoothly refined grid we proceed as follows:
\begin{enumerate}
 \item If a cell $C_i^n$ which is to be refined has a neighbour $C_j^n,~ j\in\{i-1, ~i+1\}$ on a lower level $L_j^n<L_i^n$ , we refine the neighbour $C_j^n$ as well and iterate this strategy with $C_j^n$.
 \item If a cell $C_i^n$ that is marked for coarsening has a neighbour $C_j^n$ on a higher level which is not marked for coarsening, we do not coarsen $C_i^n$.
\end{enumerate}
\begin{figure}
\subfigure[Refinement controled by the gradient of $c$.]{
\includegraphics[width=1\textwidth]{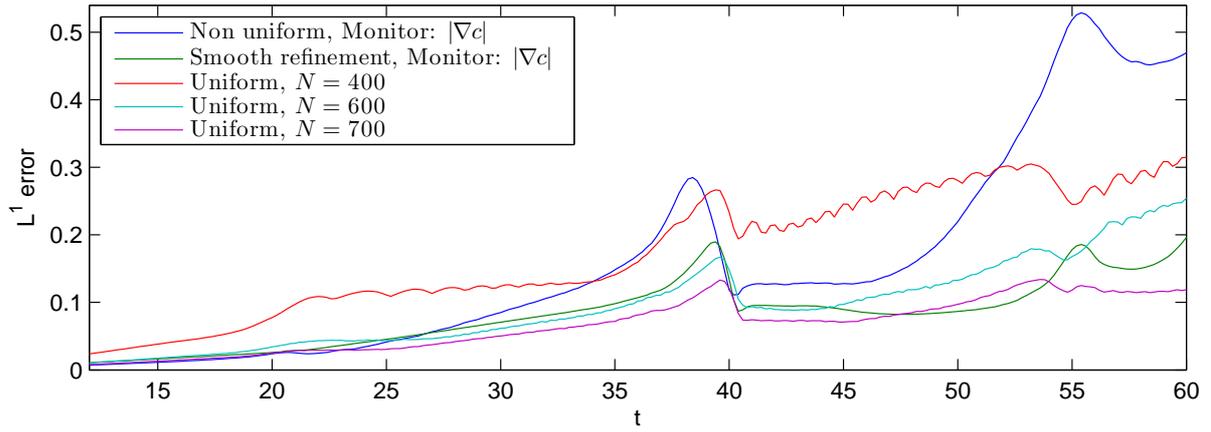}
}
\subfigure[Refinement controled by an estimation of the discretization error of the characteristic velocities.]{
\includegraphics[width=1\textwidth]{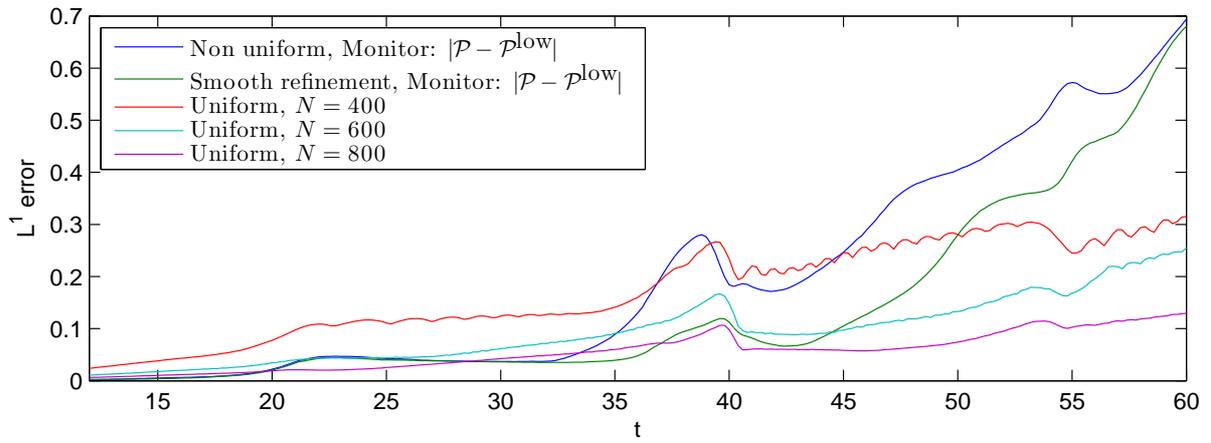}
}
\caption{Discrete $L^1$-errors as functions in time of solutions for experiment I until $t=60$. Benefits of a smoother refinement.}
\label{smooth_refinement}
\end{figure}

Figure \ref{smooth_refinement} shows a further reduction of the discrete $L^1$ errors when combining this strategy with the \textit{h-refinement} methods. The gradient controlled smooth refinement method produces solutions close to those on uniform grids with $700$ cells even for $t>40$ while the discretization error controled smooth refinement method performs comparably to an uniform solution on $800$ cells for $t\leq 45$. Reconstructing the first process of merging before $t=40$ could be significantly improved by employing smooth refinement however the error increases for later times. In the case of the smooth refinement method controlled by the discretization error, the error becomes as high as the error of its counterpart without smooth refinement. This happens at the end of the observed time interval $t=60$. However the gradient controlled smooth refinement method performs better than the uniform method on $600$ cells at $t=60$.

\begin{table}
\centering
\begin{tabular}{|c|  c || >{$} c <{$}|>{$} c <{$}|>{$} c <{$}|>{$} c <{$}|}
\hline
\multirow{ 2}{*}{Monitor} & smooth& \multirow{2}{*}{$\max \{ N_i,~t^i \leq 35\}$} & \multirow{ 2}{*}{$\sum_{i=1}^{m,~t^{m}=35} \frac{N_i}{m}$} & \multirow{ 2}{*}{$\max \{ N_i,~t^i \leq 60\}$} & \multirow{ 2}{*}{$\sum_{i=1}^{n,~t^{n}=60} \frac{N_i}{n}$}\\
& refinement & & & & \\ \hline
$|\nabla c|$ & no  & 484 & 443.6 & 603 & 471.1\\
$|\nabla c|$ & yes &500 & 455.2 & 597 & 484.2 \\
$|\mathcal{P}-\mathcal{P^\text{low}}|$ & no & 498 & 460.9 & 632 & 493.6 \\
 $|\mathcal{P}-\mathcal{P^\text{low}}|$ & yes & 481 & 449.6 &  648 & 478.5 \\
 \hline
\end{tabular}
\caption{Maximal number of cells and average number of cells used by different refinement methods for $t\in[0,35]$ and $t\in[0,60]$.}
\label{table_refinement}
\end{table}
\subsection{A 2D experiment}
We present results of a 2D-simulation, which has been conducted using IMEX3 on $\Omega=[-15,15]^2$. We employ uniform cells with grid size $h=(0.05,\, 0.05)^T$, the parameter set $\mathcal{P}$ \eqref{params} and the following initial conditions,
\begin{equation*}
  y(x) = \left\{
\begin{aligned}
&4 + 0.7 \sin(0.9x),~&x<0,\\
&7\sin(0.9 x)+0.008 x^3+4,~&0\leq x \leq 5,\\
&5+0.7\sin(4.5)+ 0.7 \sin(0.9(x-5)),~&x>5,
\end{aligned}
\right.
\end{equation*}
\begin{equation}
\left\{
\begin{aligned}
c_0(\x) &= \mathbf{1}_{\{x_2 \geq y(x_1)\}}(\x),\\
v_0(\x) &= 1-c_0(\x),\\
u_0(\x) &= 0.5c_0(\x),\\
p_0(\x) &= 0.05c_0(\x),\\
m_0(\x) &= 0,
 \end{aligned}
\right.
\end{equation}
for all $\x=(x_1,x_2)^T\in\Omega$. We display results in the window $\bar \Omega=[0,5]^2$, which includes $100 \times 100$ grid cells, since this domain stays untouched by reflections, which are caused by the homogeneous Neumann
boundary conditions, while $0\leq t\leq 200$.
\begin{figure}
\centering
\subfigure[Cancer $c$]{
\includegraphics[width=0.4\textwidth]{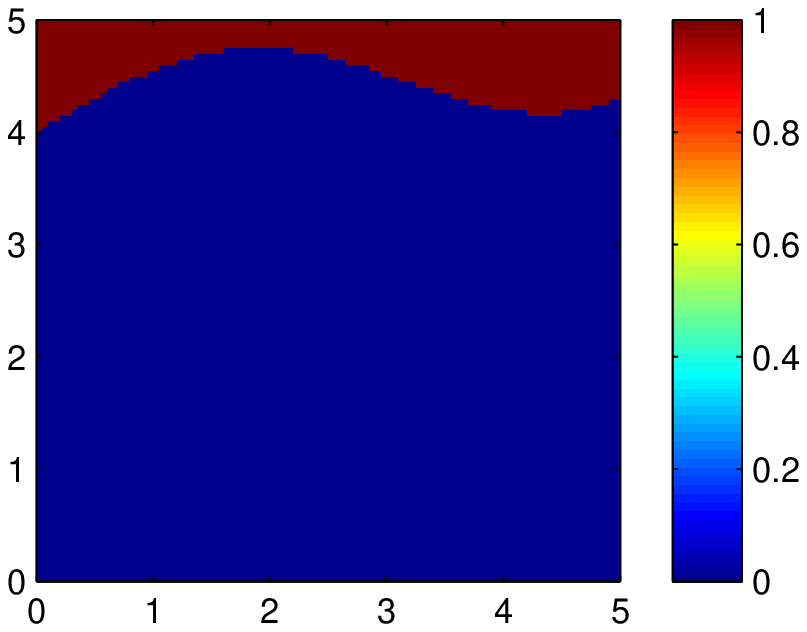}
}
\subfigure[Vitronectin $v$]{
\includegraphics[width=0.4\textwidth]{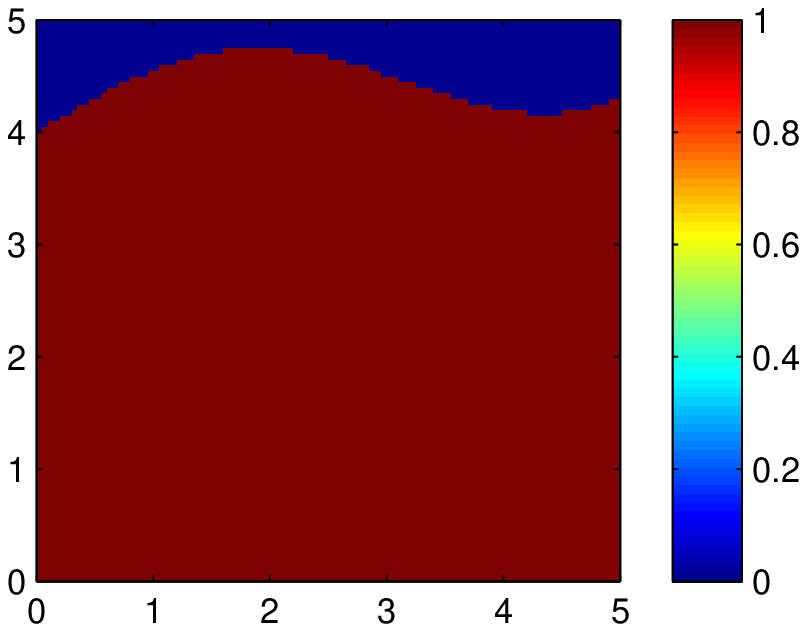}
}
\subfigure[Urokinase $u$]{
\includegraphics[width=0.25\textwidth]{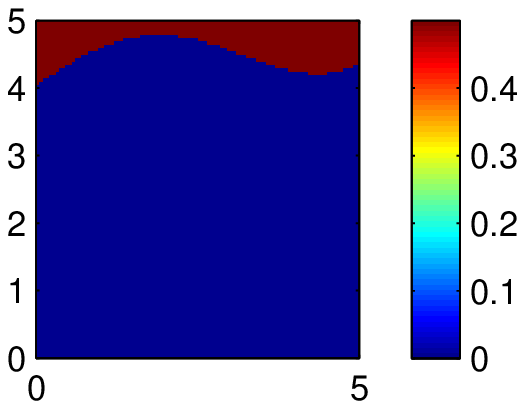}
}
\subfigure[PAI-1 $p$]{
\includegraphics[width=0.25\textwidth]{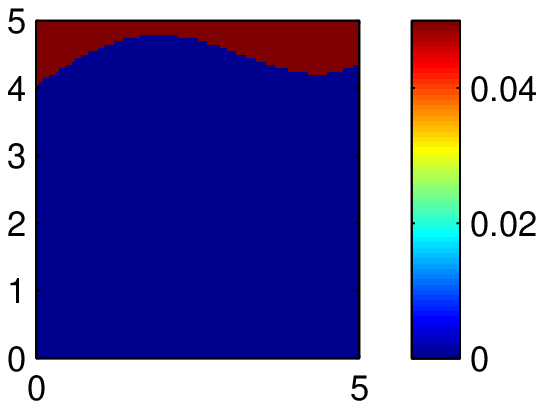}
}
\subfigure[Plasmin $m$]{
\includegraphics[width=0.25\textwidth]{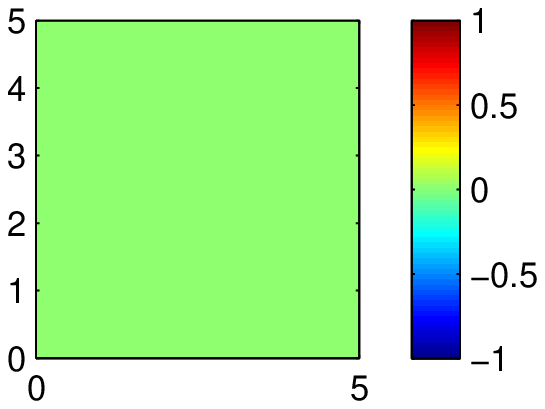}
}
\caption{Initial conditions of the 2D experiment}
\label{exp3T0fig}
\end{figure}

\begin{figure}
\centering
\subfigure[Cancer $c$]{
\includegraphics[width=0.4\textwidth]{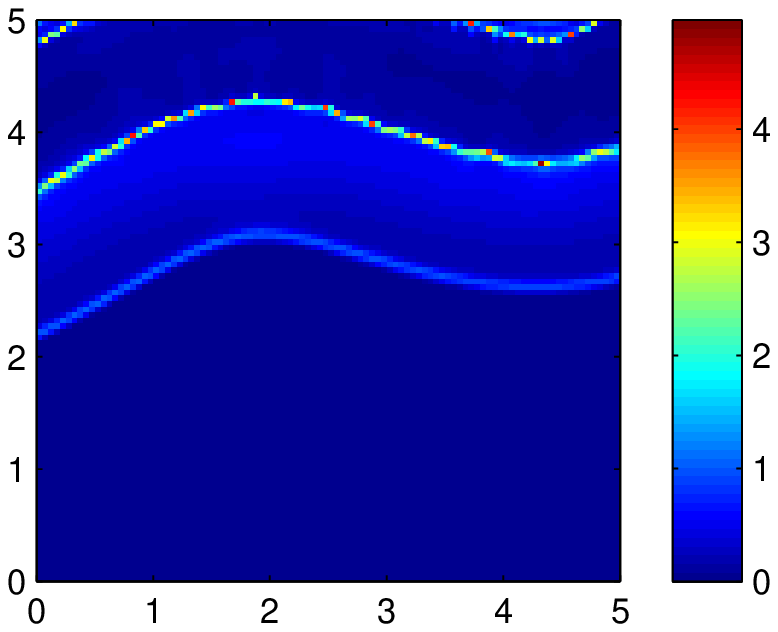}
}
\subfigure[Vitronectin $v$]{
\includegraphics[width=0.4\textwidth]{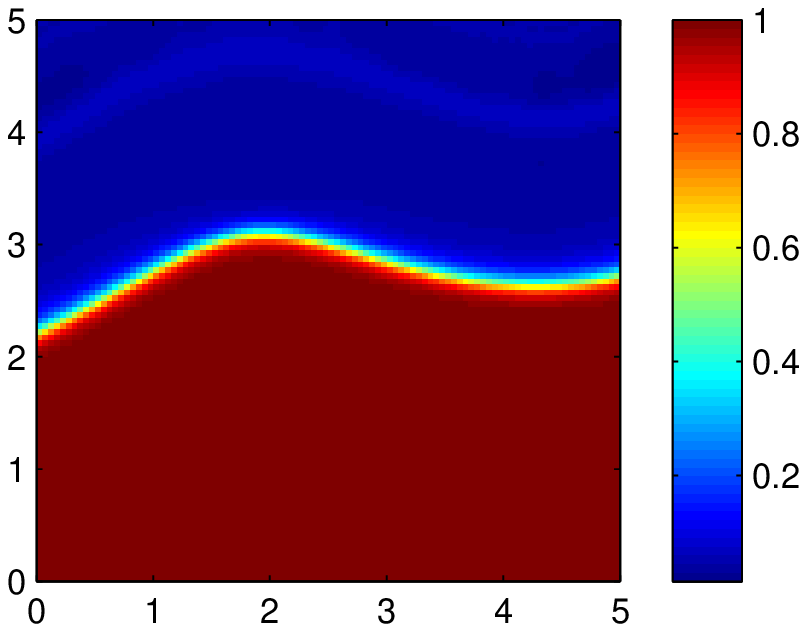}
}
\subfigure[Urokinase $u$]{
\includegraphics[width=0.25\textwidth]{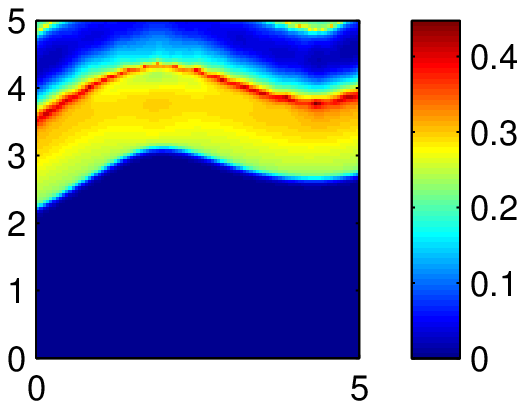}
}
\subfigure[PAI-1 $p$]{
\includegraphics[width=0.25\textwidth]{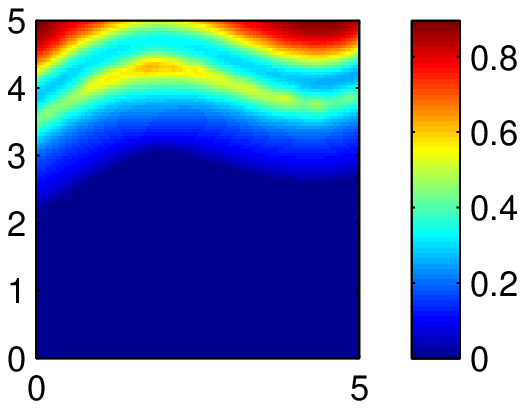}
}
\subfigure[Plasmin $m$]{
\includegraphics[width=0.25\textwidth]{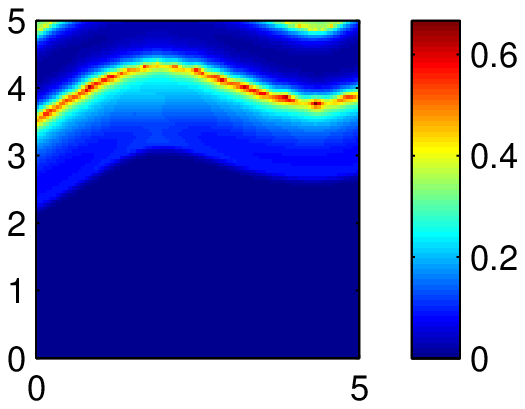}
}
\caption{2D experiment at $t=50$: degeneration of the ECM by a front of highly concentrated cancer cells.}
\label{exp3T50fig}
\end{figure}

The initial conditions are visualized in Figure \ref{exp3T0fig}. Subsequently, the accumulated cancer cells disseminate and degenerate the ECM. They  travel in negative $x_2$ direction, towards high densities of vitronectin. Formations of heterogeneous patterns of cancer cells, occurs at areas where the ECM has already degenerated by the propagating cancer cell front. No steady states have appeared until $T=200$. The corresponding results are shown in Figures \ref{exp3T50fig}, \ref{exp3T100fig}, and \ref{exp3T200fig}.
\begin{figure}
\centering
\subfigure[Cancer $c$]{
\includegraphics[width=0.4\textwidth]{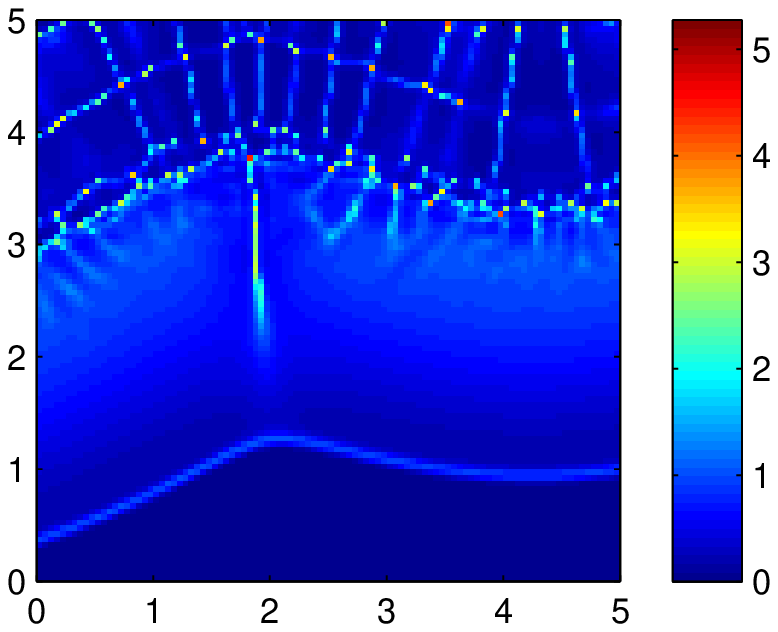}
}
\subfigure[Vitronectin $v$]{
\includegraphics[width=0.4\textwidth]{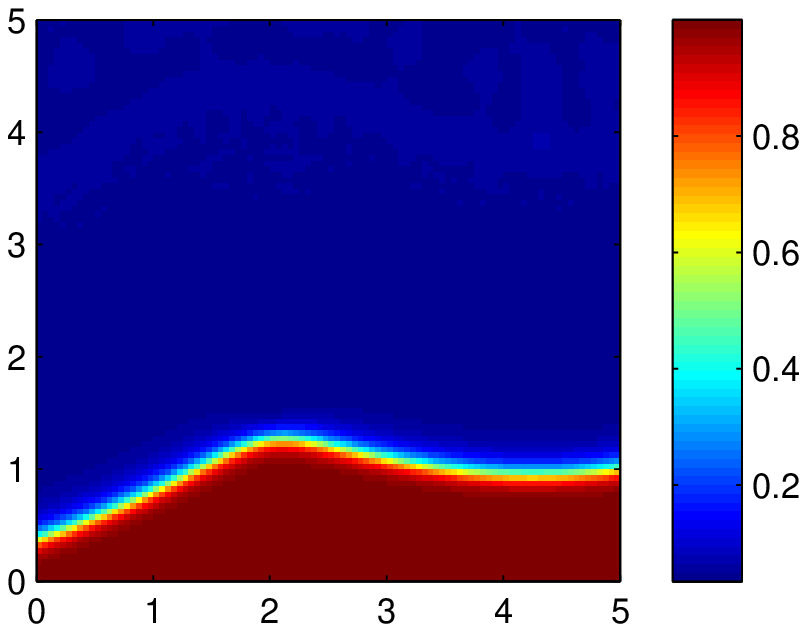}
}
\subfigure[Urokinase $u$]{
\includegraphics[width=0.25\textwidth]{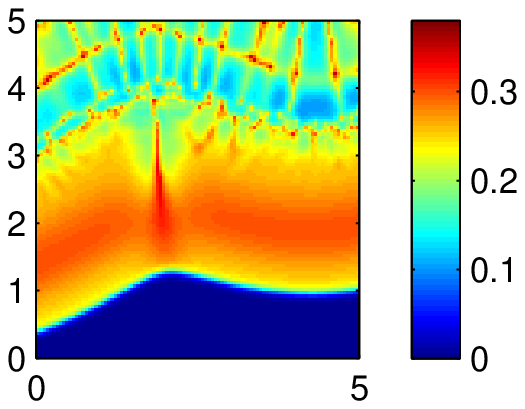}
}
\subfigure[PAI-1 $p$]{
\includegraphics[width=0.25\textwidth]{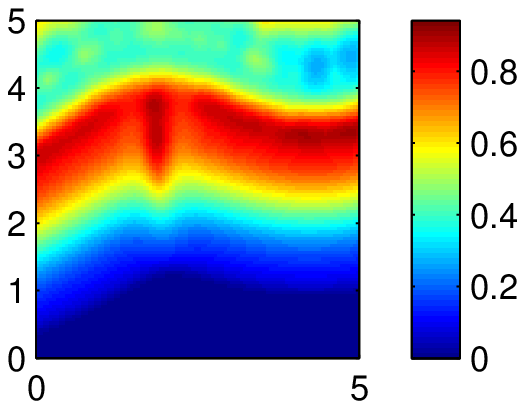}
}
\subfigure[Plasmin $m$]{
\includegraphics[width=0.25\textwidth]{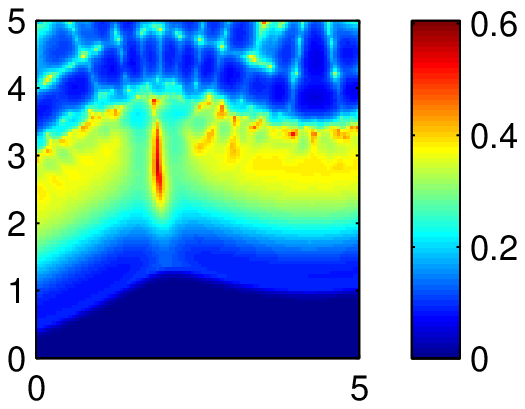}
}
\caption{2D experiment at $t=100$: formation of clusters in the cancer cell densities.}
\label{exp3T100fig}
\end{figure}
\begin{figure}
\centering
\subfigure[Cancer $c$]{
\includegraphics[width=0.4\textwidth]{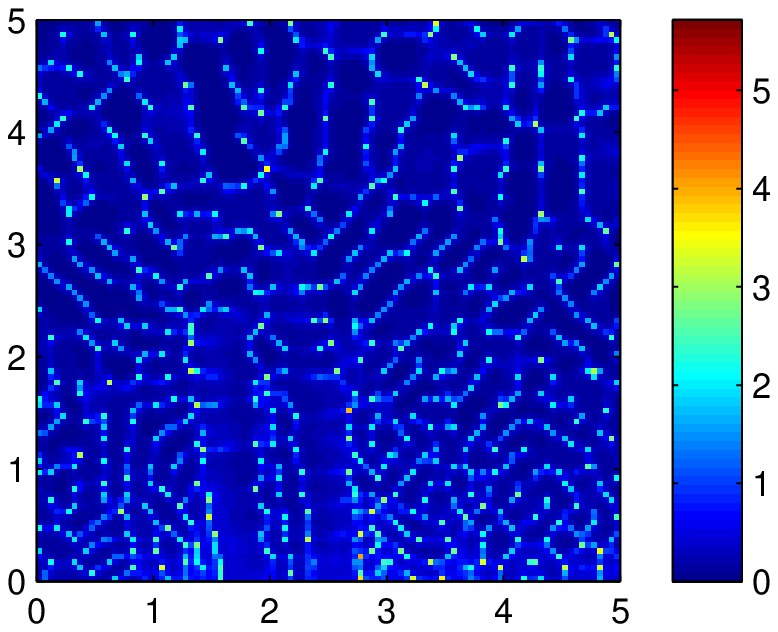}
}
\subfigure[Vitronectin $v$]{
\includegraphics[width=0.4\textwidth]{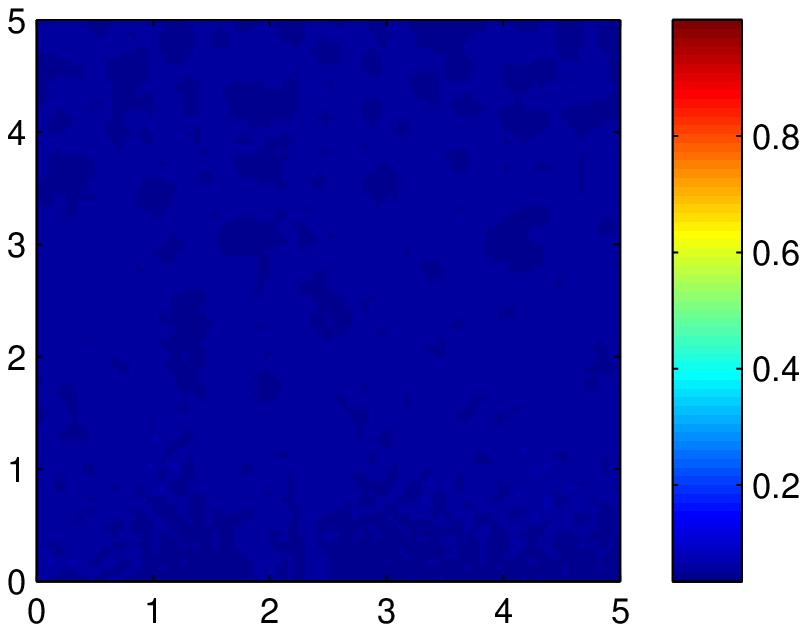}
}
\subfigure[Urokinase $u$]{
\includegraphics[width=0.25\textwidth]{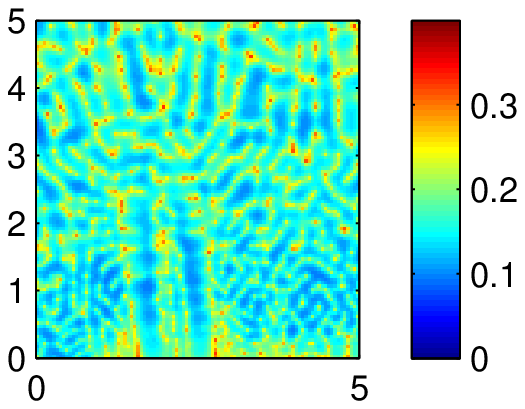}
}
\subfigure[PAI-1 $p$]{
\includegraphics[width=0.25\textwidth]{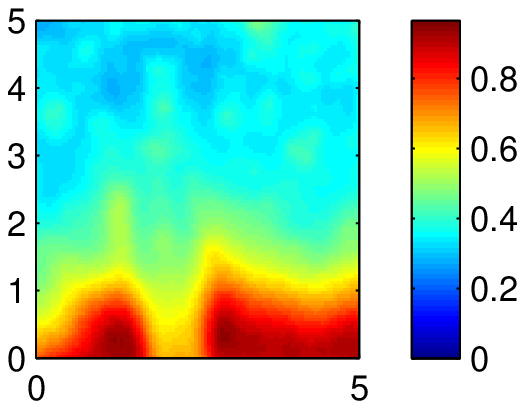}
}
\subfigure[Plasmin $m$]{
\includegraphics[width=0.25\textwidth]{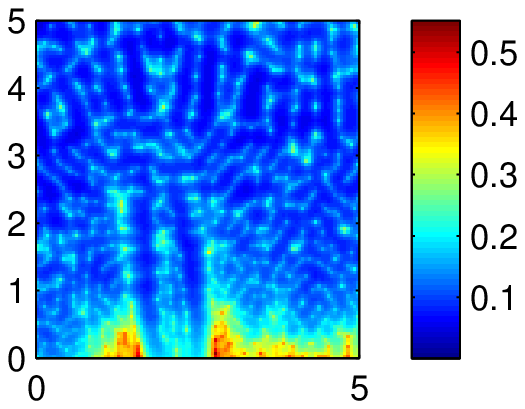}
}
\caption{2D experiment at $t=200$: formation of pattern in the cell and enzyme density structure.}
\label{exp3T200fig}
\end{figure}

\section*{Conclusion}
We address in this work the formation cancer cell clusters and the dynamics of the cancer cell invasion of the ECM. The model \eqref{chaplolsystem} we use, has been proposed in \cite{chaplain2005mathematical} and features the role of the serine protease uPA.

We employ, in Section \ref{methods}, a higher order finite volume method able to resolve the dynamics of the solution of the system \eqref{chaplolsystem}. Choosing the IMEX3 method for the time integration we observe, in Section \ref{experiments}, an experimental second order of convergence. However, even this high order method necessitates very fine discretization grids in order to produce accurate results. 

In Section \ref{experiments} we demonstrate  that these computational costs can be reduced by employing mesh refinement techniques, in particular h-refinement/cell bisection. We have noticed, with a series of test scenarios, that the best results are obtained if a) the gradient of the cancer cells is used as estimator function for the refinement/coarsening, and b) the grid is smoothly refined in the sense that neighbouring cells have a difference in refinement level at most $1$.

Analytically, we have studied in Section \ref{analysis} a reduced chemotaxis-haptotaxis model with logistic growth \eqref{small_system}, which we have compared to the original system \eqref{chaplolsystem} in the following sense: as in the case of the system \eqref{chaplolsystem} with parameter set $\mathcal{P}$ \eqref{params}, so in the case of the system \eqref{small_system}, we have found parameters for which the perturbation modes grow due to chemotaxis. We have thus, justified that the system \eqref{small_system} exhibits the same phenomena of merging and emerging concentration. We have confirmed this behaviour also numerically.

We have proved $L^\infty$ bounds on the solutions of both systems and this allows us to suggest/propose the smaller model \eqref{small_system}, and the corresponding parameter set, as a test case for the further development of mesh refinement techniques for the cancer invasion models. This test case can also be used for the extension of the mesh refinement technique that we have employed here, to the two dimensional case; a task that we will take upon on a subsequent paper.

As noted earlier, this is the first in a series of papers dealing with the invasion of cancer cells on the ECM under different chemical interaction pathways. In our future study we concentrate on one particular type of cancer and make our cancer-growth model more specific. One application  that we have in mind would be the breast cancer, which is of a solid nature and its growth behaviour is quite well-understood and documented.

\bigskip
{\bf Acknowledgements.} We gratefully acknowledge the support of the Center of Computational Sciences and the Internal University Research Funding of the University of Mainz.  N. Sfakianakis wishes also to acknowledge the support of the Alexander von Humboldt Foundation.

\bibliographystyle{alpha}
\bibliography{paper_aktuell}

\newcommand{\etalchar}[1]{$^{#1}$}
\def\cprime{$'$}
\begin{thebibliography}{WWR{\etalchar{+}}94}

\bibitem[ACN{\etalchar{+}}00]{Anderson.2000}
A.R.A. Anderson, M.A.J. Chaplain, E.L. Newman, R.J.C. Steele, and A.M.
  Thompson.
\newblock Mathematical modelling of tumour invasion and metastasis.
\newblock {\em Comput. Math. Method. M.}, 2(2):129--154, 2000.

\bibitem[AD54]{Armitage.1954}
P.~Armitage and R.~Doll.
\newblock {{T}he age distribution of cancer and a multi-stage theory of
  carcinogenesis}.
\newblock {\em Brit. J. Cancer}, 8(1):1--12, Mar 1954.

\bibitem[AGL{\etalchar{+}}11]{andasari2011mathematical}
V.~Andasari, A.~Gerisch, G.~Lolas, A.P. South, and M.A.J. Chaplain.
\newblock Mathematical modeling of cancer cell invasion of tissue: biological
  insight from mathematical analysis and computational simulation.
\newblock {\em J. Math. Biol.}, 63(1):141--171, 2011.

\bibitem[AL87]{Alt.1985}
W.~Alt and D.A. Lauffenburger.
\newblock {{T}ransient behavior of a chemotaxis system modelling certain types
  of tissue inflammation}.
\newblock {\em J. Math. Biol.}, 24(6):691--722, 1987.

\bibitem[AMS10]{Sfakianakis.2010}
Ch. Arvanitis, Ch. Makridakis, and N.~Sfakianakis.
\newblock Entropy conservative schemes and adaptive mesh selection for
  hyperbolic conservation laws.
\newblock {\em J. Hyperbol. Diff. Eq.}, 2010.

\bibitem[BCH{\etalchar{+}}13]{Deutsch.2013}
K.~Boettger, A.~Chauviere, H.~Hatzikirou, E.~Schroeck, B.~Klink, and Deutsch A.
\newblock Mathematical modelling of cancer growth and treatment.
\newblock {\em Springer Lecture Notes in Mathematics Biosciences}, 2013.

\bibitem[CKWW12]{chertock2012chemotaxis}
A.~Chertock, A.~Kurganov, X.~Wang, and Y.~Wu.
\newblock On a chemotaxis model with saturated chemotactic flux.
\newblock {\em Kinet. Relat. Models}, 5(1):51--95, 2012.

\bibitem[CL05]{chaplain2005mathematical}
M.A.J. Chaplain and G.~Lolas.
\newblock Mathematical modelling of cancer cell invasion of tissue: the role of
  the urokinase plasminogen activation system.
\newblock {\em Math. Models Methods Appl. Sci.}, 15(11):1685--1734, 2005.

\bibitem[DOS04]{katuchova17}
G.P. Dunn, L.J. Old, and R.D. Schreiber.
\newblock {{T}he immunobiology of cancer immunosurveillance and immunoediting}.
\newblock {\em Immunity}, 21(2):137--148, Aug 2004.

\bibitem[DQ12]{Deryugina.2012}
E.I. Deryugina and J.P. Quigley.
\newblock Cell surface remodeling by plasmin: a new function for an old enzyme.
\newblock {\em J. Biomed. Biotechnol.}, 2012.

\bibitem[Fis58]{Fisher.1958}
J.~C. Fisher.
\newblock {{M}ultiple-mutation theory of carcinogenesis}.
\newblock {\em Nature}, 181(4609):651--652, Mar 1958.

\bibitem[FZS{\etalchar{+}}06]{Frieboes.2006}
H.B. Frieboes, X.~Zheng, C.H. Sun, B.~Tromberg, R.~Gatenby, and V.~Cristini.
\newblock {{A}n integrated computational/experimental model of tumor invasion}.
\newblock {\em Cancer Res.}, 66(3):1597--1604, Feb 2006.

\bibitem[GC06]{gerisch2006robust}
A.~Gerisch and M.A.J. Chaplain.
\newblock Robust numerical methods for taxis-diffusion-reaction systems:
  applications to biomedical problems.
\newblock {\em Math. Comput. Modelling}, 43(1-2):49--75, 2006.

\bibitem[GC08]{Gerisch.2008}
A.~Gerisch and M.A.J. Chaplain.
\newblock Mathematical modelling of cancer cell invasion of tissue: Local and
  non-local models and the effect of adhesion.
\newblock {\em J. Theor. Biol.}, 250(4):684 -- 704, 2008.

\bibitem[GFJ{\etalchar{+}}11]{Gupta.2011}
P.B. Gupta, C.M. Fillmore, G.~Jiang, S.D. Shapira, K.~Tao, C.~Kuperwasser, and
  E.S. Lander.
\newblock {{S}tochastic state transitions give rise to phenotypic equilibrium
  in populations of cancer cells}.
\newblock {\em Cell}, 146(4):633--644, Aug 2011.

\bibitem[GV02]{gerisch2002operator}
A.~Gerisch and J.~G. Verwer.
\newblock Operator splitting and approximate factorization for
  taxis-diffusion-reaction models.
\newblock {\em Appl. Numer. Math.}, 42(1-3):159--176, 2002.

\bibitem[HD02]{katuchova1}
A.~Ho and S.F. Dowdy.
\newblock {{R}egulation of {G}(1) cell-cycle progression by oncogenes and tumor
  suppressor genes}.
\newblock {\em Curr. Opin. Genet. Dev.}, 12(1):47--52, Feb 2002.

\bibitem[HNW93]{hairer1991solving}
E.~Hairer, S.P. N{\o}rsett, and G.~Wanner.
\newblock {\em Solving ordinary differential equations. {I}}.
\newblock Springer Series in Computational Mathematics. Springer-Verlag,
  Berlin, second edition, 1993.

\bibitem[HV03]{hundsdorfer2003numerical}
W.~Hundsdorfer and J.~Verwer.
\newblock {\em Numerical solution of time-dependent
  advection-diffusion-reaction equations}, volume~33 of {\em Springer Series in
  Computational Mathematics}.
\newblock Springer-Verlag, Berlin, 2003.

\bibitem[KBKR12]{jana}
J.~Katuchova, J.~Bober, V.~Katuch, and J.~Radonak.
\newblock {{S}ignificance of Lymph Node Micrometastasis in Pancreatic Cancer
  Patients}.
\newblock {\em Eur. Sur. Res.}, 48(1):10--15, Jan 2012.

\bibitem[KC03]{christopher2001additive}
C.A. Kennedy and M.H. Carpenter.
\newblock Additive {R}unge-{K}utta schemes for convection-diffusion-reaction
  equations.
\newblock {\em Appl. Numer. Math.}, 44(1-2):139--181, 2003.

\bibitem[KLM14]{kurganovnumerical}
A.~Kurganov and M.~Luk\'a\v{c}ov\'a-Medvi\v{d}ov\'a.
\newblock Numerical study of two-species chemotaxis models.
\newblock {\em Discrete Cont. Dyn-B}, 19(1):131--152, 2014.

\bibitem[KO00]{kroner2000posteriori}
D.~Kr{\"o}ner and M.~Ohlberger.
\newblock A posteriori error estimates for upwind finite volume schemes for
  nonlinear conservation laws in multi dimensions.
\newblock {\em Math. Comp.}, 69(229):25--39, 2000.

\bibitem[Kol13]{kolbe2013master}
N.~Kolbe.
\newblock Mathematical {M}odeling and {N}umerical {S}imulations of {C}ancer
  {I}nvasion.
\newblock Master's thesis, Johannes Gutenberg-Universit\"at Mainz, 2013.

\bibitem[KPW10]{Kessenbrock.2010}
K.~Kessenbrock, V.~Plaks, and Z.~Werb.
\newblock Matrix metalloproteinases: regulators of the tumor microenvironment.
\newblock {\em Cell}, 2010.

\bibitem[KS71]{keller1971model}
E.F. Keller and L.A. Segel.
\newblock Model for chemotaxis.
\newblock {\em J. Theor. Biol.}, 30(2):225--234, 1971.

\bibitem[KWC72]{katuchova4}
J.F. Kerr, A.H. Wyllie, and A.R. Currie.
\newblock {{A}poptosis: a basic biological phenomenon with wide-ranging
  implications in tissue kinetics}.
\newblock {\em Brit. J. Cancer}, 26(4):239--257, Aug 1972.

\bibitem[LMS13]{Sfakianakis.2013a}
M.~Luk\'a\v{c}ov\'a-Medvi\softd{o}v\'a and N.~Sfakianakis.
\newblock Entropy dissipation of moving mesh adaptation.
\newblock {\em (accepted) J. Hyperbol. Diff. Eq.}, 2013.

\bibitem[Lol03]{lolas2003phd}
G.~Lolas.
\newblock {\em Mathematical modelling of the urokinase plasminogen activation
  system and its role in cancer invasion of tissue}.
\newblock PhD thesis, University of Dundee, 2003.

\bibitem[LSU68]{ladyzhenskaeiia1968linear}
O.A. Ladyzenskaja, V.A. Solonnikov, and N.N. Uralceva.
\newblock {\em Linear and quasi-linear equations of parabolic type}.
\newblock AMS, 1968.

\bibitem[MCP10]{Czochra.2010}
A.~Marciniak-Czochra and M.~Ptashnyk.
\newblock Boundedness of solutions of a haptotaxis model.
\newblock {\em Math. Models Methods Appl. Sci.}, 2010.

\bibitem[Nor53]{Nordling.1953}
C.O. Nordling.
\newblock {{A} new theory on cancer-inducing mechanism}.
\newblock {\em Brit. J. Cancer}, 7(1):68--72, Mar 1953.

\bibitem[Ohl99]{ohlberger19999adaptive}
M.~Ohlberger.
\newblock Adaptive mesh refinement for single and two phase flow problems in
  porous media.
\newblock In {\em Finite volumes for complex applications {II}}, pages
  761--768. Hermes Sci. Publ., Paris, 1999.

\bibitem[OM04]{katuchova2}
H.~Okada and T.W. Mak.
\newblock {{P}athways of apoptotic and non-apoptotic death in tumour cells}.
\newblock {\em Nat. Rev. Cancer}, 4(8):592--603, Aug 2004.

\bibitem[PAG{\etalchar{+}}09]{Poplawski}
N.J. Poplawski, U.~Agero, J.S. Gens, M.~Swat, J.A. Glazier, and A.R.A.
  Anderson.
\newblock {{F}ront instabilities and invasiveness of simulated avascular
  tumors}.
\newblock {\em Bull. Math. Biol.}, 71(5):1189--1227, Jul 2009.

\bibitem[PAS10]{Painter.2010}
K.J. Painter, N.J. Armstrong, and J.A. Sherratt.
\newblock {{T}he impact of adhesion on cellular invasion processes in cancer
  and development}.
\newblock {\em J. Theor. Biol.}, 264(3):1057--1067, Jun 2010.

\bibitem[Pat53]{patlak1953random}
C.S. Patlak.
\newblock Random walk with persistence and external bias.
\newblock {\em Bull. Math. Biophys.}, 15:311--338, 1953.

\bibitem[PH11]{painter2011spatio}
K.J. Painter and T.~Hillen.
\newblock Spatio-temporal chaos in a chemotaxis model.
\newblock {\em Physica D}, 240(4):363--375, 2011.

\bibitem[PR05]{pareschi2005implicit}
L.~Pareschi and G.~Russo.
\newblock Implicit-{E}xplicit {R}unge-{K}utta schemes and applications to
  hyperbolic systems with relaxation.
\newblock {\em J. Sci. Comput.}, 25(1-2):129--155, 2005.

\bibitem[Pre03]{Preziosi.2003}
L.~Preziosi.
\newblock {\em Cancer modelling and simulation}.
\newblock CRC Press, 2003.

\bibitem[PS11]{puppo2012adaptive}
G.~Puppo and M.~Semplice.
\newblock Numerical entropy and adaptivity for finite volume schemes.
\newblock {\em Commun. Comput. Phys.}, 2011.

\bibitem[PSNB96]{Perumpanani.1996}
A.J. Perumpanani, J.A. Sherratt, J.~Norbury, and H.M. Byrne.
\newblock {{B}iological inferences from a mathematical model for malignant
  invasion}.
\newblock {\em Invas. Metast.}, 16(4-5):209--221, 1996.

\bibitem[RCP11]{Condeelis.2011}
E.T. Roussos, J.S. Condeelis, and A.~Patsialou.
\newblock Chemotaxis in cancer.
\newblock {\em Nat. Rev. Cancer}, 11(8):573--587, 2011.

\bibitem[SCW{\etalchar{+}}94]{Seiffert.1994}
D.~Seiffert, G.~Ciambrone, N.V. Wagner, B.R. Binder, and D.J. Loskutoff.
\newblock The somatomedin b domain of vitronectin. structural requirements for
  the binding and stabilization of active type 1 plasminogen activator
  inhibitor.
\newblock {\em J. Biol. Chem.}, 1994.

\bibitem[Sfa13]{Sfakianakis.2013b}
N.~Sfakianakis.
\newblock Adaptive mesh reconstruction for hyperbolic conservation laws with
  total variation bound.
\newblock {\em Math. Comp.}, 2013.

\bibitem[SMC12]{Czochra.2012}
T.~Stiehl and A.~Marciniak-Czochra.
\newblock Mathematical modeling of leukemogenesis and cancer stem cell
  dynamics.
\newblock {\em Math. Mod. Nat. Phen.}, 7:166--202, 1 2012.

\bibitem[Spo96]{katuchova5}
M.B. Sporn.
\newblock {{T}he war on cancer}.
\newblock {\em Lancet}, 347(9012):1377--1381, May 1996.

\bibitem[SRLC09]{Szymanska.2009}
Z.~Szyma{\'n}ska, C.M. Rodrigo, M.~Lachowicz, and M.A.~J. Chaplain.
\newblock Mathematical modelling of cancer invasion of tissue: the role and
  effect of nonlocal interactions.
\newblock {\em Math. Models Methods Appl. Sci.}, 2009.

\bibitem[Tao09]{tao2009global}
Y.~Tao.
\newblock Global existence of classical solutions to a combined
  chemotaxis-haptotaxis model with logistic source.
\newblock {\em J. Math. Anal. Appl.}, 354(1):60--69, 2009.

\bibitem[TGE08]{katuchova3}
M.C.B. Tan, P.S. Goedegebuure, and T.J. Eberlein.
\newblock Tumor biology and tumor markers.
\newblock In {\em Sabiston Textbook of Surgery, The Biological Basis of Modern
  Surgical Practice}, volume~18. Saunders, 2008.

\bibitem[TS02]{Turner.2002}
S.~Turner and J.A. Sherratt.
\newblock {{I}ntercellular adhesion and cancer invasion: a discrete simulation
  using the extended {P}otts model}.
\newblock {\em J. Theor. Biol.}, 216(1):85--100, May 2002.

\bibitem[TSL00]{tyson2000fractional}
R.~Tyson, L.G. Stern, and R.J. LeVeque.
\newblock Fractional step methods applied to a chemotaxis model.
\newblock {\em J. Math. Biol.}, 41(5):455--475, 2000.

\bibitem[VL77]{van1977towards}
B.~Van~Leer.
\newblock Towards the ultimate conservative difference scheme. {IV}. {A} new
  approach to numerical convection.
\newblock {\em J. Comput. Phys.}, 23(3):276--299, 1977.

\bibitem[Wrz04]{wrzosek2004global}
D.~Wrzosek.
\newblock Global attractor for a chemotaxis model with prevention of
  overcrowding.
\newblock {\em Nonlinear Anal.}, 59(8):1293--1310, 2004.

\bibitem[WWR{\etalchar{+}}94]{Wei.1994}
Y.~Wei, D.A. Waltz, N.~Rao, R.J. Drummond, S.~Rosenberg, and H.A. Chapman.
\newblock Identification of the urokinase receptor as an adhesion receptor for
  vitronectin.
\newblock {\em J. Biol. Chem.}, 1994.

\bibitem[Zlo06]{katuchova15}
A.~Zlotnik.
\newblock {{C}hemokines and cancer}.
\newblock {\em Int. J. Cancer}, 119(9):2026--2029, 2006.

\end{thebibliography}
\end{document}